\documentclass[twoside,a4paper,10pt]{article}
\usepackage[all]{xy}

\usepackage{graphicx}
\usepackage[T1]{fontenc}
\usepackage[latin1]{inputenc}
\usepackage{amsmath,amsthm}
\usepackage{amsfonts,amssymb}
\usepackage{bigdelim}
\usepackage{multirow}
\usepackage{enumerate}
\usepackage{stmaryrd}
\usepackage{float}
\usepackage{pstricks}
\usepackage[ae]{}
\newcommand{\E}{\mathbb{E}}
\newcommand{\N}{\mathbb{N}}
\newcommand{\Z}{\mathbb{Z}}

\newcommand{\R}{\mathbb{R}}

\newcommand{\Pb}{\mathbb{P}}

\newcommand{\ve}{\varepsilon}

\def\={{\;\mathop{=}\limits^{\text{(law)}}\;}}

\newtheorem{theorem}{Theorem}[section]
\newtheorem{prop}[theorem]{Proposition}
\newtheorem{lemma}[theorem]{Lemma}
\newtheorem{defi}[theorem]{Definition}
\newtheorem{corol}[theorem]{Corollary}

\theoremstyle{definition}
\newtheorem{rem}[theorem]{Remark}

\newtheorem{counterexa}[theorem]{Counterexample}
\newtheorem{exa}[theorem]{Example}

\numberwithin{equation}{section}

\usepackage{fancyhdr}
\date{}

\usepackage[english]{babel}
\title{An application of multivariate total positivity to peacocks}
\author{\small ANTOINE MARIE BOGSO \footnote{Universit\'e de Lorraine, Institut Elie Cartan de Lorraine, UMR 7502, Vandoeuvre-l\`es-Nancy, F-54506, France.
CNRS, Institut Elie Cartan de Lorraine, UMR 7502, Vandoeuvre-l\`es-Nancy, F-54506, France. Email addresses: antoine.bogso@univ-lorraine.fr, ambogso@gmail.com}
}
\begin{document}
\maketitle
 
{\footnotesize  {\bf Abstract:} $\;$
We use multivariate total positivity theory to exhibit new families of peacocks. As the authors of \cite{HPRY}, our guiding example is the result of Carr-Ewald-Xiao \cite{CEX}. We shall introduce the notion of strong conditional monotonicity. This concept is strictly more restrictive than the conditional monotonicity as defined in \cite{HPRY} (see also \cite{Be}, \cite{BPR1} and \cite{ShS1}). There are many random vectors which are strongly conditionally monotone (SCM). Indeed, we shall prove that multivariate totally positive of order 2 (MTP$_2$) random vectors are SCM. As a consequence, stochastic processes with MTP$_2$ finite-dimensional marginals are SCM. This family includes processes with independent and log-concave increments, and  one-dimensional diffusions which have absolutely continuous transition kernels.}\\
{\footnotesize  {\bf Key words:} $\;$ convex order, peacocks, total positivity of order 2 (TP$_2$), multivariate total positivity of order 2 (MTP$_2$), Markov property, strong conditional monotonicity.}

\chead[\small ANTOINE MARIE BOGSO]{An application of multivariate total positivity to peacocks}

\section{Introduction}
We call peacock a real valued process $(Z_t,t\geq0)$ which is integrable, i.e. $\E[|Z_t|]<\infty$ for every $t\geq0$, and which is increasing in the convex order, i.e., for every convex fonction $\psi:\R\to\R$,
\begin{equation}\label{eq:DefPcoc}
\text{the map: }t\in\R_+\longmapsto\E[\psi(Z_t)]\in]-\infty,+\infty]\text{ is non-decreasing.}
\end{equation}
Observe that if $(Z_t,t\geq0)$ is a peacock, then $\E[Z_t]$ does not depend on $t$. Indeed, it suffices to apply (\ref{eq:DefPcoc}) first with $\psi(x)=x$, then with $\psi(x)=-x$. 
The pun {\em peacock} comes from the french: ``\underline{P}rocessus \underline{C}roissant pour l'\underline{O}rdre \underline{C}onvexe" which acronym: ``P.C.O.C" may be pronounced ``peacock".
To prove that an integrable process satisfies (\ref{eq:DefPcoc}), it suffices to consider convex functions which belong to the set:
$$
\mathbf{C}:=\{\psi:\R\to\R\text{ convex }\mathcal{C}^2\text{-function such that }\psi^{\prime\prime}\text{ has a compact support}\},
$$
where $\psi^{\prime\prime}$ denotes the second order derivative of $\psi$.
Note that if $\psi\in\mathbf{C}$, then its derivative $\psi^{\prime}$ is bounded and there exist $k_1,k_2\geq0$ such that:
$$
\forall\,x\in\R,\,\text{ }|\psi(x)|\leq k_1+k_2|x|.
$$
There are two remarkable results which motivate the investigation on peacocks. In 2008, Carr-Ewald-Xiao \cite{CEX} proved that Asian options are increasing in the convex order with respect to the maturity. In other words, if $(B_s,s\geq0)$ denotes a standard Brownian motion issued from $0$, then
\begin{equation}\label{eq:CEX}\tag{CEX08}
\left({\bf N}_t:=\frac{1}{t}\int_0^t e^{B_s-\frac{s}{2}}ds,t\geq0\right)\text{ is a peacock.}
\end{equation}
The second result due to Kellerer \cite{Kel} states that a real valued process $(Z_t,t\geq0)$ is a peacock if and only if there exists a martingale $(M_t,t\geq0)$ with the same one-dimensional marginals as $(Z_t,t\geq0)$, i.e., for every $t\geq0$, $M_t\=Z_t$. This martingale is not unique in general and it may be chosen Markovian. Recently, Hisch-Roynette \cite{HR} offered a new proof of Kellerer's theorem. However, the Kellerer's proof is not constructive, and then it helps neither establishing whether or not a process is a peacock, nor constructing an {\em associated} martingale to a given peacock $(Z_t,t\geq0)$, i.e. a martingale having the same one-dimensional marginals as $(Z_t,t\geq0)$. In \cite{BYa}, Baker-Yor provide an associated martingale to $({\bf N}_t,t\geq0)$ using the Brownian sheet. Inspired by Carr-Ewald-Xiao and Baker-Yor results, the authors of \cite{HPRY} exhibited several examples of peacocks and they provided  several methods to associate explicitely  martingales to certain of them. We refer the reader to 
\cite{RS}, \cite{HK}, \cite{BDMY}, \cite{BYb}  and  \cite{BPR2} for further interesting results about peacock processes.\\
In this paper, we exibit new families of peacocks using multivariate total positivity theory. But for many of them, finding an associated  martingale remains open. Let us mention that total positivity is a nice property that plays an important role in various domains of Mechanics and Mathematics. There is a large amount of literature concerning total positivity. We shall follow Karlin \cite{KAa} and Karlin-Rinot \cite{KR} for basic definitions and results.\\ In section 2, we  give some basic results concerning total positivity and multivariate total positivity of order 2. Section 3 is reserved to strong  conditional monotonicity results. Finally, in section 4, we use strong conditional monotonicity theorems to exhibit new classes of peacocks, inspired from the Carr-Ewald-Xiao example.

\section{Total positivity and multivariate total positivity \\of order 2}
We first define totally positive functions of order 2 and give several examples of Markov processes with totally positive transition kernels. Then, we deal with an extension of total positivity of order 2 to multivariate distributions.
\subsection{Totally positive functions of order 2}
We follow the terminology and notation of Karlin \cite{KA}.
\begin{defi}
A function $p:\R\times\R\to\R_+$ is said to be {\em totally positive of order 2} (TP$_2$) if for every real numbers $x_1<x_2$, $y_1<y_2$,
\begin{equation}\label{eq:TP2cont}\tag{TP$_2$}
p\left(
\begin{array}{cc}
x_1,x_2\\
y_1,y_2
\end{array}
\right):=\det\left(
\begin{array}{cc}
p(x_1,y_1)&p(x_1,y_2)\\ &\\
p(x_2,y_1)&p(x_2,y_2)
\end{array}
\right)\geq0.
\end{equation}
Similarly, a function $p:\Z\times\Z\to\R_+$ is said to be TP$_2$ if, for every integers $k_1<k_2$ and $l_1<l_2$, $p$ satisfies (\ref{eq:TP2cont}). 
\end{defi}
\noindent
Note that one may define totally positive functions of order higher than 2 (see \cite{KA}).
\begin{rem}\label{rem:TP2SetD}
Let $D$ denote a subset of $\R\times\R$ which satisfies the following property:
\begin{equation}\label{eq:SetD}\tag{P}
\left.
\begin{array}{l}
\text{For every }x_1<x_2\text{ and }y_1<y_2,\\ \\
\left[(x_1,y_2)\in D\text{ and }(x_2,y_1)\in D\right] \Longrightarrow
\left[(x_1,y_1)\in D\text{ and }(x_2,y_2)\in D\right].
\end{array}
\right\}
\end{equation}
Let $p:D\to\R_+$ be TP$_2$, i.e. for every $x_1<x_2$, $y_1<y_2$ such that $(x_1,y_1)$, $(x_1,y_2)$, $(x_2,y_1)$ and $(x_2,y_2)$ belong to $D$,
$$
p\left(
\begin{array}{cc}
x_1,x_2\\
y_1,y_2
\end{array}
\right)\geq0.
$$
We define:
$$
\forall\,x,y\in\R,\,\text{ }\widehat{p}(x,y)=
\left\{
\begin{array}{cl}
p(x,y)&\text{if }(x,y)\in D,\\
0&\text{otherwise.}
\end{array}
\right.
$$
Then, $\widehat{p}$ is TP$_2$ if and only if $p$ is TP$_2$.\\
Here are some examples of $D\subset \R\times\R$ satisfying (\ref{eq:SetD}).
\begin{enumerate}
\item[i)]If $I$ and $J$ are two intervals of $\R$, then $I\times J$ satisfies (\ref{eq:SetD}).
\item[ii)]For every reals $k_0<k_1$ and every $(\alpha,\beta)\in\R\times\R\setminus\{(0,0)\}$,
$$
D=\{(x,y)\in\R\times\R;\,k_0\leq\alpha x-\beta y\leq k_1\}\text{ satifies }(\ref{eq:SetD}).
$$
\end{enumerate}
\end{rem}
We  give properties of  TP$_2$ functions assuming  that they are defined on $\R\times\R$. By the preceding remark, one may extend these results to functions defined on subsets of $\R\times\R$ which satisfy (\ref{eq:SetD}).\\
The following characterization result of smooth TP$_2$ functions is proved in \cite{KAa}.
\begin{prop}\label{prop:SmoothTP2}(Karlin \cite{KAa}).
Let $p:\R\times\R\to\R_+$ be such that the partial derivatives $\dfrac{\partial p}{\partial x}$,
$\dfrac{\partial p}{\partial y}$, and $\dfrac{\partial^2 p}{\partial x\partial y}$ exist at each point $(x,y)$ of $\R\times\R$.
\item[1)]If $p$ is TP$_2$, then, for every reals $x_1<x_2$ and $y$,
\begin{equation}\label{eq:SmoothTP2a}
\det\left(
\begin{array}{cc}
p(x_1,y)&\dfrac{\partial p}{\partial y}(x_1,y)\\ &\\
p(x_2,y)&\dfrac{\partial p}{\partial y}(x_2,y)
\end{array}
\right)\geq0,
\end{equation}
and, for every $(x,y)\in\R\times\R$,
\begin{equation}\label{eq:SmoothTP2b}
\det\left(
\begin{array}{cc}
p(x,y)&\dfrac{\partial p}{\partial y}(x,y)\\ &\\
\dfrac{\partial p}{\partial x}(x,y)&\dfrac{\partial^2 p}{\partial x\partial y}(x,y)
\end{array}
\right)\geq0.
\end{equation}
\item[2)]Conversely, if $p(x,y)>0$ for every $(x,y)\in\R\times\R$, then (\ref{eq:SmoothTP2b}) implies (\ref{eq:SmoothTP2a}), which in turn implies that $p$ is TP$_2$.
\end{prop}
A second criterion of smooth TP$_2$ functions follows from Proposition \ref{prop:SmoothTP2}.
\begin{corol}\label{corol:SmoothTP2}(Karlin \cite{KAa}).
Let $p:\R\times\R\to\R_+$ be strictly positive and such that at each point $(x,y)$ of $\R\times\R$, the second order partial derivative $\dfrac{\partial^2(\log p)}{\partial x\partial y}(x,y)$ exists. Then $p$ is TP$_2$ if and only if
$$
\frac{\partial^2(\log p)}{\partial x\partial y}(x,y)\geq0.
$$ 
\end{corol}
\begin{proof}
This follows from (\ref{eq:SmoothTP2b}) and from the straightforward relation:
$$
p^2(x,y)\frac{\partial^2(\log p)}{\partial x\partial y}(x,y)=\det\left(
\begin{array}{cc}
p(x,y)&\dfrac{\partial p}{\partial y}(x,y)\\ &\\
\dfrac{\partial p}{\partial x}(x,y)&\dfrac{\partial^2p}{\partial x\partial y}(x,y)
\end{array}
\right).
$$
\end{proof}
\begin{exa}(Brownian transition densities).\\
Consider the family $(p_t:\R\times\R\to\R_+,t>0)$ given by:
$$
\forall\,(x,y)\in\R\times\R,\,\text{ }p_t(x,y)=\frac{1}{\sqrt{2\pi t}}\exp\left(\frac{-(x-y)^2}{2t}\right).
$$
For every $t>0$,
$$
\frac{\partial^2(\log p_t)}{\partial x\partial y}(x,y)=\frac{1}{t}>0,
$$
and, by Corollary \ref{corol:SmoothTP2}, $p_t$ is TP$_2$.\\
More generally, if $f:\R\to\R_+$ is a strictly positive $\mathcal{C}^2$-function, then $(x,y)\longmapsto f(x-y)$ is TP$_2$ if and only if $f$ is log-concave. Indeed,
$$
\frac{\partial^2}{\partial x\partial y}[\log f(x-y)]=-(\log f)^{\prime\prime}(x-y).
$$
\end{exa}
\begin{exa}(Ornstein-Uhlenbeck transiton densities).\\
Let $(p_t,t>0)$ be the densities defined on $\R\times\R$ by:
$$
p_t(x,y)=\sqrt{\frac{ce^{ct}}{2\pi\sinh(ct)}}\exp\left(-ce^{ct}\frac{(y-xe^{-ct}-\nu(1-e^{-ct}))^2}{2\sinh(ct)}\right)\,\text{ }(c,\,\nu\in\R).
$$
For every $t>0$,
$$
\frac{\partial^2(\log p_t)}{\partial x\partial y}(x,y)=\frac{c}{\sinh(ct)}>0,
$$
and, from Corollary \ref{corol:SmoothTP2}, $p_t$ is TP$_2$.
\end{exa}
\begin{counterexa}
The function $p$ defined by:
$$
\forall\,x,y\in\R,\,\text{ }p(x,y)=\frac{1}{1+(x-y)^2}
$$
is not TP$_2$.
\end{counterexa}
New TP$_2$ functions are generated using the following classical composition formula:
\begin{lemma}\label{lem:ConvProduct}
Let $p,q:\R\times\R\to\R$ be two Borel functions, and let $\sigma$ be a positive measure on $\R$ such that:
$$
\forall\,x,z\in\R,\text{ }\int_{\R}|p(x,y)||q(y,z)|\,\sigma(dy)<\infty.
$$
Let $r$ denote the function defined on $\R\times\R$ by:
$$
\forall\,x,z\in\R,\text{ }r(x,z)=\int_{\R}p(x,y)q(y,z)\sigma(dy).
$$
Then, for every reals $x_1<x_2$, $z_1<z_2$,
\begin{equation}\label{eq:ConvProduct}
r\left(
\begin{array}{c}
x_1,x_2\\
z_1,z_2
\end{array}
\right)=
\iint_{y_1<y_2}p\left(
\begin{array}{c}
x_1,x_2\\
y_1,y_2
\end{array}
\right)
q\left(
\begin{array}{c}
y_1,y_2\\
z_1,z_2
\end{array}
\right)
\sigma(dy_1)\sigma(dy_2).
\end{equation}
\end{lemma}
\begin{rem}\label{rem:ConvProduct}
If $p$ and $q$ are two integrable functions with respect to the Lebesgue measure and if $r:=p\ast q$ denotes the convolution product of $p$ and $q$, i.e.
$$
\forall\,x\in\R,\text{ }r(x)=\int_{\R}p(x-y)q(y)dy,
$$
then, using (\ref{eq:ConvProduct}), for every $x_1<x_2$ and $z_1<z_2$,
\begin{align*}
&\det\left(
\begin{array}{cc}
r(x_1-z_1)&r(x_1-z_2)\\ &\\
r(x_2-z_1)&r(x_2-z_2)
\end{array}
\right)\\
&=\iint_{y_1<y_2}\det\left(
\begin{array}{cc}
p(x_1-y_1)&p(x_1-y_2)\\ &\\
p(x_2-y_1)&p(x_2-y_2)
\end{array}
\right)\det\left(
\begin{array}{cc}
q(y_1-z_1)&q(y_1-z_2)\\&\\
q(y_2-z_1)&q(y_2-z_2)
\end{array}
\right)dy_1dy_2.
\end{align*} 
\end{rem}
From Lemma \ref{lem:ConvProduct} and Remark \ref{rem:ConvProduct}, we easily deduce the following result:
\begin{prop}\label{prop:ConvProduct}
Let $p,q:\R\times\R\to\R_+$ be two TP$_2$ functions such that the product $r:\R\times\R\to\R_+$ given by:
$$
\forall\,x,z\in\R,\text{ }r(x,z)=\int_{\R}p(x,y)q(y,z)dy
$$
is finite. Then $r$ is TP$_2$.\\ In particular, if $p$ and $q$ are two integrable log-concave functions, then the convolution product $r=p\ast q$ is also log-concave. 
\end{prop}
\begin{rem}
Proposition \ref{prop:ConvProduct} allows to regularise TP$_2$ functions in such a way that total positivity is preserved. Indeed, if $q$ is a TP$_2$ function such that, for every $z\in\R$, $q(\cdot,z)$ is integrable, then, for every $\ve>0$, the function $q_{\ve}$ defined by:
$$
\forall\,x,z\in\R,\text{ }q_{\ve}(x,z)=\frac{1}{\ve\sqrt{2\pi}}\int_{\R}\exp\left[-\frac{(x-y)^2}{2\ve^2}\right]q(y,z)dy
$$
is TP$_2$. Moreover,
$$
\forall\,z\in\R,\text{ }\lim\limits_{\ve\to 0}q_{\ve}(\cdot,z)=q(\cdot,z)\,\text{ in }L^{1}(\R).
$$
\end{rem}
 
\subsection{Markov processes with totally positive transition kernels}
In this paragraph,  we present some examples of Markov processes with totally positive transition kernels.
\begin{defi}
Let $P:=(P_{s,t}(x,dy),0\leq s<t,x\in I)$ be the transition function of a Markov process $((X_t,t\geq0),(\Pb_x,x\in I))$ with values in a sub-interval $I$ of $\R$. $P$ is said to be {\em totally positive of order 2} (TP$_2$) if, for every $0\leq s<t$, every $x_1<x_2$ elements of $I$, and every Borel subsets $E_1$, $E_2$ of $I$ such that $E_1<E_2$ (i.e. $a_1<a_2$ for every $a_1\in E_1$ and $a_2\in E_2$), we have:
\begin{equation}\label{eq:defMarkovTP2}
P_{s,t}\left(
\begin{array}{c}
x_1,x_2\\
E_1,E_2
\end{array}
\right):=\det\left(
\begin{array}{cc}
P_{s,t}(x_1,E_1)&P_{s,t}(x_1,E_2)\\ &\\
P_{s,t}(x_2,E_1)&P_{s,t}(x_2,E_2)
\end{array}
\right)\geq0.
\end{equation}
Suppose moreover that $(X_t,t\geq0)$ is time-homogeneous. Then $P$ is TP$_2$ if and only if 
\begin{equation}\label{eq:defHmgeneMarkovTP2}
P_{t}\left(
\begin{array}{c}
x_1,x_2\\
E_1,E_2
\end{array}
\right):=\det\left(
\begin{array}{cc}
P_{t}(x_1,E_1)&P_{t}(x_1,E_2)\\ &\\
P_{t}(x_2,E_1)&P_{t}(x_2,E_2)
\end{array}
\right)\geq0.
\end{equation} 
\end{defi}
\begin{rem}
Let $P:=(P_{s,t}(x,dy),0\leq s<t,x\in I)$ denote the transition function of a Markov process $((X_t,t\geq0),(\Pb_x,x\in I))$ taking values in a sub-interval $I$ of $\R$. We suppose that, for every $0\leq s<t$ and $x\in I$, $P_{s,t}(x,dy)$ has a continuous density $p_{s,t}(x,\cdot)$ with respect to a $\sigma$-finite regular measure. Then $P$ is TP$_2$ if and only if $p_{s,t}$ is TP$_2$, i.e. for every $x_1<x_2$, $y_1<y_2$ elements of $I$,
$$
p_{s,t}\left(
\begin{array}{c}
x_1,x_2\\
y_1,y_2
\end{array}
\right):=\det\left(
\begin{array}{cc}
p_{s,t}(x_1,y_1)&p_{s,t}(x_1,y_2)\\&\\
p_{s,t}(x_2,y_1)&p_{s,t}(x_2,y_2)
\end{array}
\right)\geq0.
$$
\end{rem}
\begin{defi}
Let $P:=(P_{s,t}(k,l),0\leq s<t,(k,l)\in I\times I)$ be the transition function of a continuous time Markov chain  $(X_t,t\geq0)$ which takes values in a sub-interval $I$ of $\Z$. We say that $P$ is TP$_2$ if, for every $0\leq s<t$ and every integers $k_1<k_2$, $l_1<l_2$ in $I$,
\begin{equation}\label{eq:defiMarkovZTP2}
P_{s,t}\left(
\begin{array}{c}
k_1,k_2\\
l_1,l_2
\end{array}
\right):=\det\left(
\begin{array}{cc}
P_{s,t}(k_1,l_1)&P_{s,t}(k_1,l_2)\\ &\\
P_{s,t}(k_2,l_1)&P_{s,t}(k_2,l_2)
\end{array}
\right)\geq0.
\end{equation}
If $(X_t,t\geq0)$ is time-homogeneous, then (\ref{eq:defiMarkovZTP2}) is equivalent to:
\begin{equation}\label{eq:defiMarkovHmgeneZTP2}
P_{t}\left(
\begin{array}{c}
k_1,k_2\\
l_1,l_2
\end{array}
\right):=\det\left(
\begin{array}{cc}
P_{t}(k_1,l_1)&P_{t}(k_1,l_2)\\ &\\
P_{t}(k_2,l_1)&P_{t}(k_2,l_2)
\end{array}
\right)\geq0.
\end{equation}
\end{defi}
There are many Markov processes with  totally positive transition kernels. Let us give some of them.
\subsubsection{Processes with independent and log-concave increments} 
Let $(X_{t},t\geq0)$ be a real valued process with independent increments, i.e.
\begin{equation}\label{eq:PAI}\tag{PII}
\forall\,0\leq s\leq t,\,\text{ }X_t-X_s\text{ is independent of }\mathcal{F}_s:=\sigma(X_u;\,0\leq u\leq s).
\end{equation}
We suppose that, for every $0\leq s<t$, the increment $X_t-X_s$ is log-concave. In other words, 
$X_t-X_s$ has a density $p_{s,t}$ (with respect to Lebesgue measure) which is log-concave, i.e.
\begin{equation}\label{eq:PaiLogConcave}
\forall\,x,y\in \R,\,\theta\in]0,1[,\text{ }p_{s,t}(\theta x+(1-\theta)y)\geq(p_{s,t}(x))^{\theta}(p_{s,t}(y))^{1-\theta}.
\end{equation}
If $X_t-X_s$ takes values in $\Z$, then (\ref{eq:PaiLogConcave}) may be replaced by
\begin{equation}\label{eq:PaiZLogConcave}
p_{s,t}^2(k)\geq p_{s,t}(k-1)p_{s,t}(k+1),
\end{equation}
where, for every $k\in\Z$, $p_{s,t}(k):=\Pb(X_t-X_s=k)$.\\
Many common r.v.'s are log-concave. Indeed, Gaussian, uniform, exponential, binomial, negative binomial, geometric and Poisson  r.v.'s are log-concave. On the contrary, Gamma r.v.'s with parameter $a\in(0,1)$ are not log-concave.
\begin{theorem}\label{theo:LogConcave}(An \cite{AN}, Daduna-Szekli\cite{DS}). A Lebesgue-measurable function $f:\R\to\R_+$ is log-concave 
 if and only if
\begin{equation}\label{eq:PaiTP2}
\forall\,x_1<x_2\in\R,\,y_1<y_2\in\R,\text{ }\det\left(
\begin{array}{cc}
f(x_1-y_1)&f(x_1-y_2)\\ &\\
f(x_2-y_1)&f(x_2-y_2)
\end{array}
\right)\geq0.
\end{equation}
\end{theorem}
\noindent
Then, $(X_t,t\geq0)$ is a Markov process with transition function $P_{s,t}(x,dy)$ given by:
\begin{equation}\label{eq:PAITransition}
\forall\,0\leq s<t,\,x\in\R,\text{ }P_{s,t}(x,dy)=p_{s,t}(y-x)dy,
\end{equation}
and, since $X_t-X_s$ is log-concave, we deduce from (\ref{eq:PAITransition}) and from Theorem \ref{theo:LogConcave} that $P_{s,t}(x,dy)$ is TP$_2$.

\subsubsection{Absolute value of a process with independent, symmetric and PF$_{\infty}$ increments}
Let $(X_t,t\geq0)$ be a real valued process with independent increments (\ref{eq:PAI}) such that, for every $0\leq s<t$, $X_t-X_s$ is symmetric and has a density denoted $p_{s,t}$. Note that $p_{s,t}$  is symmetric, i.e., for every $x\in\R_+$, $p_{s,t}(x)=p_{s,t}(-x)$. Then, $(|X_t|,t\geq0)$ is a Markov process whose transition function is given by:
$$
P_{s,t}(x,dy):=p^{\ast}_{s,t}(x,y)dy\,\text{ }(0\leq s<t,\,x\in\R),
$$
where
$$
p^{\ast}_{s,t}(x,y)=p_{s,t}(y-x)+p_{s,t}(-x-y).
$$
We suppose in addition that, for every $0\leq s<t$, $p_{s,t}$ is a P\'olya frequency (PF$_{\infty}$) function, i.e. for every integers $r\geq 1$, $1\leq m\leq r$, and every $x_1<x_2<\cdots<x_{m}$, $y_1<y_2<\cdots<y_m$,
\begin{equation}\label{eq:PF}\tag{PF$_{\infty}$} 
\det[p_{s,t}(x_i-y_j);\,1\leq i,j\leq m]\geq0.
\end{equation}
Then the transition function $P_{s,t}(x,dy)$ is TP$_2$, i.e., for every $0\leq s<t$, $p^{\ast}_{s,t}$ is TP$_2$. This is a direct consequence of the following result due to Karlin \cite{KA}.
\begin{theorem}(Karlin \cite{KA}).\label{theo:PFsymTP2}
Let $f:\R\to\R_+$ be a symmetric PF$_{\infty}$ density function. Then  the function $f^{\ast}$  defined on $\R_+\times\R_+$ by
$$
f^{\ast}(x,y)=f(x-y)+f(-x-y)
$$
is a TP$_2$ function. 
\end{theorem}
We mention that Theorem \ref{theo:PFsymTP2} remains valid if we consider discrete and symmetric PF$_{\infty}$ densities. To prove Theorem \ref{theo:PFsymTP2}, the author applies Schoenberg's characterisation of symmetric PF$_{\infty}$ densities in terms of their Laplace transforms. 
\begin{theorem}(Schoenberg \cite{Sch}).\label{theo:PFcontCharact}
A symmetric density function $f:\R\to\R_+$ is PF$_{\infty}$ if and only if its Laplace transform $\Phi:\,s\longmapsto\displaystyle\int_{-\infty}^{\infty} e^{-sy}f(y)dy$ exists in a strip (of the complex plane) including the imaginary axis in its interior and has the form
\begin{equation}\label{eq:PFrealForm}
\Phi(s)=\frac{e^{\alpha s^2}}{\prod\limits_{i=1}^{\infty}(1-a_i^2s^2)},
\end{equation} 
where $\alpha\geq0$, $a_i\in\R$ for every $i$, and $0<\alpha+\sum\limits_{i=1}^{\infty}a_i^2<\infty$.
\end{theorem}
A discrete analog of Theorem \ref{theo:PFcontCharact} has been proved by Edrei \cite{ED}.
\begin{theorem}(Edrei \cite{ED}).\label{theo:PFdiscreteCharact}
A symmetric density function $f:\Z\to\R_+$ is PF$_{\infty}$ if and only if its Laurent series
$S(z)=\displaystyle\sum\limits_{k=-\infty}^{\infty}f(k)z^k$ converges in some ring (of the complex plane) including the unit circle in its interior, and the analytic continuation $\widetilde{S}$ of $S$ is of the form
\begin{equation}\label{eq:PFdiscreteForm}
\widetilde{S}(z)=C\exp\left(a\left(z+z^{-1}\right)\right) \frac{\prod\limits_{i=1}^{\infty} \left((1+\alpha_i^2)+\alpha_i\left(z+\dfrac{1}{z}\right)\right)}{\prod\limits_{i=1}^{\infty}(1-\gamma_iz)(1-\gamma_iz^{-1})},
\end{equation}
where $C\geq0$, $a\geq0$, $\alpha_i\geq0$, $0\leq \gamma_i<1$ for every $i$, and $\sum\limits_{i=1}^{\infty}(\alpha_i+\gamma_i)<\infty$.
\end{theorem}
Here are some examples of symmetric PF$_{\infty}$ density functions.
\begin{exa}
By Theorem \ref{theo:PFcontCharact},
\item[i)] a symmetric gaussian density is PF$_{\infty}$,
\item[ii)] for every $\lambda>0$, the density function $f(x)=\frac{\lambda}{2}\exp(-\lambda|x|)$ is PF$_{\infty}$.
\end{exa}
\begin{exa}Theorem \ref{theo:PFdiscreteCharact} applies in the following cases:
\item[i)]For every $\alpha\in\R_+$, the density function $f:\Z\to\R_+$ defined by
$$
f(0)=\frac{1+\alpha^2}{(1+\alpha)^2},\,\text{ }\,f(1)=f(-1)=\frac{\alpha}{(1+\alpha)^2},
$$
and $f(k)=0$ if $k\notin\{-1,0,1\}$ is PF$_{\infty}$. Indeed, if we define
$$
S(z):=\sum\limits_{k=-\infty}^{\infty}f(k)z^k,
$$
then
$$
S(z)=\frac{1}{(1+\alpha)^2}\left[(1+\alpha^2)+\alpha\left(z+\frac{1}{z}\right)\right]
$$
which is of the form (\ref{eq:PFdiscreteForm}).
\item[ii)]Let $a>0$ and $c=\dfrac{(1-e^a)^2}{1-e^{-2a}}$. The density function $f:\Z\to\R_+$ given by 
$$
f(0)=c\,\text{ and }\,f(k)=c\,e^{-a|k|},\,\text{ }k\in\Z\setminus\{0\}
$$
is PF$_{\infty}$, since its Laurent series $\sum\limits_{k=-\infty}^{\infty}f(k)z^k$ admits the representation
$$
\widetilde{S}(z)=\frac{(1-e^{-a})^2}{(1-e^{-a}z)\left(1-\dfrac{e^{-a}}{z}\right)}
$$
which is of the form (\ref{eq:PFdiscreteForm}).
\end{exa}

\subsubsection{One-dimensional diffusions}
An important class of real valued Markov processes with TP$_2$ transition kernels consists of one-dimensional diffusions.
\begin{theorem}\label{theo:TimeHomoDiffuTP2}(Karlin-Taylor \cite{KaT}, Chapter 15, Problem 21).\\
Let $((X_t,t\geq0),(\Pb_x,x\in I))$ be a one-dimensional diffusion on a sub-interval $I$ of $\R$, and let $(P_t(x,dy),t\geq0,x\in I)$ be its transition function. We suppose that $(t,x)\longmapsto P_t(x,dy)$ is continuous in $x$ for every $t$. Then, for every $t\geq0$, $P_t(x,dy)$ is TP$_2$, i.e. for every $x_1,x_2\in I$ and every Borel subsets $E_1<E_2$ of $I$,
\begin{equation}\label{eq:DiffuMarkovTP2}
\det\left(
\begin{array}{cc}
P_t(x_1,E_1)&P_t(x_1,E_2)\\ &\\
P_t(x_2,E_1)&P_t(x_2,E_2)
\end{array}
\right)\geq0.
\end{equation}
In particular, if $P_t(x,dy)=p_t(x,y)dy$, and if $p_t$ is continuous in $y$ for every $x$, then  
$p_t$ is a TP$_2$ function.
\end{theorem}
This result is due at origin to Karlin-McGregor \cite{KaM1} who showed that (\ref{eq:DiffuMarkovTP2}) has a probabilistic interpretation: {\em Suppose that two  particles $q_1$ and $q_2$, started at time zero in states $x_1$ and $x_2$ respectively, execute the process $(X_t,t\geq0)$ simultaneously and independently. Then the determinant 
\begin{equation}\label{eq:DetHomoStrongMarkov}
 \det\left(
\begin{array}{cc}
P_t(x_1,E_1)&P_t(x_1,E_2)\\&\\
P_t(x_2,E_1)&P_t(x_2,E_2)
\end{array}
\right)
\end{equation}
is equal to the probability that at time $t$, $q_1$ is located in $E_1$ and $q_2$ is located in $E_2$ without these particles having occupied simultaneously a common state at some earlier time $\tau<t$.} \\ Karkin-McGregor \cite{KaM1} reach the same result for several time-homogeneous and strong Markov processes whose state space is a subset of the real line. In particular, the transition probability matrix of a birth and death process is TP$_2$.\\
Moreover, as a consequence of Theorem \ref{theo:TimeHomoDiffuTP2}, bridges of  one-dimensional diffusions have  TP$_2$ transition functions. We refer to \cite{FPY} for a rigorous definition of the bridge of a one-dimensional diffusion.

\subsection{Multivariate total positivity of order 2}
\begin{defi}\label{defi:MTP2}
\item[1)]A function $p:\R^n\to\R_+$ ($n\in\N$, $n\geq 2$) is said to be multivariate totally positive of order 2 (MTP$_2$) if for every $\mathbf{x}=(x_1,\cdots,x_n)$ and $\mathbf{y}=(y_1,\cdots,y_n)$ in $\R^n$,
\begin{equation}\label{eq:MTP2}\tag{MTP$_2$}
p(\mathbf{x}\wedge \mathbf{y})p(\mathbf{x}\vee \mathbf{y})\geq p(\mathbf{x})p(\mathbf{y}),
\end{equation}
where 
\begin{equation}
\begin{array}{ccc}
& &\mathbf{x}\wedge \mathbf{y}=(\min(x_1, y_1),\cdots,\min(x_n,y_n))\\
and & &\\
& &\mathbf{x}\vee \mathbf{y}=(\max(x_1,y_1),\cdots,\max(x_n,y_n)).
\end{array}
\end{equation}
\item[2)]A random vector $(X_1,\cdots,X_n)$ with real components is said to be multivariate totally positive of order 2 (MTP$_2$) if it is absolutely continuous with respect to a $\sigma$-finite product measure (which we shall always denote $dx_1\cdots dx_n$) and if its density $p:\R^n\to\R_+$ is MTP$_2$.
\end{defi}
\begin{rem}\label{rem:MTP2}
By definition, the MTP$_2$ property is invariant under permutations, i.e. if a random vector $(X_1,\cdots,X_n)$ is  MTP$_2$, then, for every permutation $\pi$ of $\{1,\cdots,n\}$, $(X_{\pi(1)},\cdots,X_{\pi(n)})$ is MTP$_2$.
\end{rem}
The following result is proved in Karlin-Rinot \cite{KR}.
\begin{theorem}\label{theo:MTP2Density}(Karlin-Rinot \cite{KR}).
\item[1)]If $(X_1,\cdots,X_n)$ is a MTP$_2$ random vector, then, for every $k\in\N$, $2\leq k\leq n$, $(X_1,\cdots,X_k)$ is MTP$_2$.
\item[2)]Let $p:\R^n\to\R_+$ be a MTP$_2$ density. Then, for every $k\in\N$, $2\leq k\leq n$, and for every continuous and bounded functions $f_i:\R\to\R_+$, $i=1,\cdots,n$, the function $p^{(k)}:\R^k\to\R_+$ given by:
\begin{align*}
&p^{(k)}(x_1,\cdots,x_k)\\
&=\prod\limits_{i=1}^kf_i(x_i)\int_{\R^{n-k}}p(x_1,\cdots,x_k,u_{k+1},\cdots,u_n)\prod\limits_{j=k+1}^nf_j(u_j)du_{k+1}\cdots du_n
\end{align*}
is MTP$_2$.
\end{theorem}
As a consequence of Theorem \ref{theo:MTP2Density}, we have:
\begin{corol}\label{corol:ConvolMTP2}(Karlin-Rinot \cite{KR}).\\ 
Let $l,m,n\in\N^{\ast}$. Let $p:\R^l\times\R^m\to\R_+$ and $q:\R^m\times\R^n\to\R_+$ be two MTP$_2$ densities. Then the function $r:\R^l\times\R^n$ defined by:
$$
\forall\,(x,z)\in\R^l\times\R^n,\text{ }r(x,z)=\int_{\R^m}p(x,y)q(y,z)dy
$$ 
is MTP$_2$.
\end{corol}
\begin{rem}
Corollary \ref{corol:ConvolMTP2} allows to regularize MTP$_2$ densities while preserving MTP$_2$ property.
\end{rem}
Note that MTP$_2$ random vectors satisfy a co-monotony principle. Indeed, Sarkar \cite{Sar} proved 
the following result.
\begin{theorem}\label{theo:MTP2Comonotony}(Sarkar \cite{Sar}).
Let $\mathbf{X}:=(X_1,\cdots,X_n)$ ($n\in\N^{\ast}$) be MTP$_2$, and let $\phi,\varphi:\R^n\to\R$ be two measurable and simultaneously (componentwise) non-decreasing (resp. non-increasing) on $\R^n$. Then
\begin{equation}\label{eq:Comonotony}
\E[\phi(\mathbf{X})\varphi(\mathbf{X})]\geq\E[\phi(\mathbf{X})]\E[\varphi(\mathbf{X})].
\end{equation}
\end{theorem}
Recently, Pag\`es \cite{Pa} introduced a functional co-monotony principle for stochastic processes $\mathbf{X}:=(X_t,t\geq0)$. He offered an extension of Inequality (\ref{eq:Comonotony}) by replacing $\phi$ and $\varphi$ with functionals of the hole path of $\mathbf{X}$. The author also provided many examples of co-monotone stochastic processes.\\
Here are some basic examples of MTP$_2$ distributions (see Karlin-Rinot \cite{KR}).
\begin{exa}(Karlin-Rinot \cite{KR}, Section 3).
\item[1)]If $X_1,\cdots,X_n$ is a sample of i.i.d. random variables, each $X_i$ having a density, then the joint density of the order statistics $X_{(1)},\cdots,X_{(n)}$ is MTP$_2$. 
\item[2)]A Gaussian random vector $(X_1,\cdots,X_n)$ with an invertible covariance matrix $\Sigma$ is MTP$_2$ if and only if the inverse matrix $\Sigma^{-1}$ of $\Sigma$ has negative off-diagonal elements.
\item[3)]Let $(X_1,\cdots,X_n)$ be a gaussian random vector with zero mean, and with an invertible covariance matrix $\Sigma$. Let $\Sigma^{-1}$ denote the inverse matrix of $\Sigma$. Then $(|X_1|,\cdots,|X_n|)$ is MTP$_2$ if and only if there exists a diagonal matrix $D$ with elements $\pm1$ such that $D\Sigma^{-1}D$ has negative off-diagonal elements.
\item[4)]Let $(X_t,t\geq0)$ be a Markov process with absolutely continuous and TP$_2$ transition kernel. Then, for every distinct elements $t_1,\cdots,t_n$ in $\R_+$, $(X_{t_1},\cdots,X_{t_n})$ is MTP$_2$. 
\end{exa}
Further examples of MTP2 distributions may be found in Karlin-Rinot \cite{KR} and Gupta-Richards \cite{GR}.
 
\section{Strong conditional monotonicity}
We introduce the notion of strong conditional monotonicity which strictly implies conditional monotonicity as defined in \cite{BPR1} and \cite{HPRY}. 
\begin{defi}\label{defi:MathcalIn}
For every $n\in\N^{\ast}$,  let $\mathcal{I}_n$ denotes the set of continuous and bounded functions $\phi:\R^n\to\R$ which are componentwise non-decreasing.
\end{defi}
\begin{defi}\label{defi:SCM}(Strong conditional monotonicity).
\item[1)]A random vector $(X_1,\cdots,X_n)$ with real components is said to be strongly conditionally monotone (SCM) if, for every $i\in\{1,\cdots,n\}$, every continuous and strictly positive functions $f_k:\R\to\R_+$, $k=1,\cdots,n$ such that:
\begin{equation}\label{eq:SCMdefIntFk}
\E\left[\prod\limits_{k=1}^nf_k(X_k)\right]<\infty,
\end{equation}
and every $\phi\in\mathcal{I}_n$, we have:
\begin{equation}\label{eq:SCM}\tag{SCM}
\left.
\begin{array}{ll}
&z\in \R\longmapsto K_{i}(n,z):=\dfrac{\E\left[\left.\phi(X_{1},\cdots,X_{n}) \prod\limits_{k=1}^nf_k(X_{k})\right|X_{i}=z\right]}{\E\left[\left. \prod\limits_{k=1}^nf_k(X_{k})\right|X_{i}=z\right]}
\\& \\
&\text{ is a non-decreasing function. }
\end{array}
\right\}
\end{equation}
\item[2)]A real valued process $(X_{\lambda},\lambda\geq0)$ is said to be strongly conditionally monotone (SCM) if its finite-dimensional marginals are SCM. 
\end{defi}
\begin{rem}\label{rem:SCM}
\item[1)]If there are subsets $I_1,\cdots,I_n$ of $\R$ such that, for every $k$, $X_k$ takes values in $I_k$, we may suppose in Definition \ref{defi:SCM} that $\phi$ is defined on $I_1\times\cdots\times I_k$, $f_k$ is defined on $I_k$, and $z\longmapsto K_i(n,z)$ is defined on $I_i$.
\item[2)]Let $(X_{\lambda},\lambda\geq0)$ be a real valued process, and let $\theta:\R_+\times\R\to\R$ be such that, for every $\lambda\geq0$, $x\longmapsto\theta(\lambda,x)$ is continuous and strictly increasing (resp. strictly decreasing). If $(X_{\lambda},\lambda\geq0)$ is SCM, then so is $(\theta(\lambda, X_{\lambda}),\lambda\geq0)$. 
\item[3)]Let $X$ be a real valued r.v. and let $\alpha:\R_+\to\R^{\ast}_+$ be non-decreasing. Then $(\alpha(\lambda)X,\lambda\geq0)$ is SCM.
\item[4)]Note that (\ref{eq:SCM}) implies the conditional monotonicity as defined in \cite{HPRY} and \cite{BPR1}. Indeed, if we take $f_k=1$ for every $k\in\N^{\ast}$, we recover the conditional monotonicity hypothesis. The converse is not true. For example, the Gamma subordinator is conditionally monotone (see \cite{BPR1}, Section 2 or \cite{HPRY}, Section 1.4), but not strongly conditionally monotone. 
\end{rem}
We exhibit an important class of SCM random vectors. Indeed, we prove that MTP$_2$ random vectors are SCM. This result extends Theorem 3.43 of \cite{Bo} to the non-Markovian case. 
\begin{theorem}\label{theo:SCMrvMTP2}
Every MTP$_2$ random vector is SCM.
\end{theorem}
\begin{proof}
Let $(X_1,\cdots,X_n)$ ($n\geq2$) be a MTP$_2$ random vector, and let $p:\R^n\to\R_+$ denote its density. By regularisation, we may assume (without loss of generality) that $p$ is continuous and strictly positive. Since MTP$_2$ property is invariant under permutations, it suffices to prove (\ref{eq:SCM}) for $i=n$.\\
Let $f_k:\R\to\R_+^{\ast}$, $k=1,\cdots,n$ be continuous and strictly positive functions satisfying (\ref{eq:SCMdefIntFk}). By truncature, we may suppose (without loss of generality) that all $f_k$ are bounded.\\
We shall prove by induction that, for every $l\in\{2,\cdots,n\}$, and every $\phi\in\mathcal{I}_l$,
\begin{equation}\label{eq:SCMInductEl}\tag{E$_l$}
\left.
\begin{array}{ll}
(z_l,\cdots,z_n)\longmapsto \dfrac{\E\left[\left.\phi(X_{1},\cdots,X_{l}) \prod\limits_{k=1}^lf_k(X_{k})\right|X_{l}=z_l,\cdots,X_n=z_n\right]}{\E\left[\left. \prod\limits_{k=1}^lf_k(X_{k})\right|X_{l}=z_l,\cdots,X_n=z_n\right]}\\ \\
\text{is componentwise non-decreasing.}
\end{array}
\right\}
\end{equation}
Observe that if $l=n$, we  recover (\ref{eq:SCM}) with $i=n$.
\item[$\bullet$]\textbf{Case $\mathbf{l=2}$. }For every $\phi\in\mathcal{I}_2$ and every $\mathbf{z}=(z_2,\cdots,z_n)\in\R^{n-1}$, we define:
\begin{align}
\mathcal{K}(2,\mathbf{z})&=\frac{\E\left[\phi(X_1,X_2)f_1(X_1)f_2(X_2)|X_2=z_2,\cdots,X_n=z_n\right]}{\E\left[f_1(X_1)f_2(X_2)|X_2=z_2,\cdots,X_n=z_n\right]}\nonumber\\
&=\frac{\displaystyle\int_{-\infty}^{\infty}\phi(x,z_2)f(x)p(x,\mathbf{z})dx}{\displaystyle\int_{-\infty}^{\infty}f(x)p(x,\mathbf{z})dx}\label{eq:defiMathcalK2}.
\end{align}
We set
$$
F_{\mathbf{z}}(x):=\frac{\displaystyle\int_{-\infty}^{x}f(y)p(y,\mathbf{z})dy}{\displaystyle\int_{-\infty}^{\infty}f(y)p(y,\mathbf{z})dy},
$$ 
so that (\ref{eq:defiMathcalK2}) may be written:
$$
\mathcal{K}(2,\mathbf{z})=\int_{-\infty}^{\infty}\phi(x,z_2)dF_{\mathbf{z}}(x)=\int_{0}^1\phi\left(F^{-1}_{\mathbf{z}}(u),z_2\right)du,
$$
since $F_{\mathbf{z}}$ is continuous and strictly increasing.\\
Therefore, for every $u\in[0,1]$, it suffices to show that 
\begin{equation}\label{eq:InverseFz}
\mathbf{z}\longmapsto F^{-1}_{\mathbf{z}}(u)\quad\text{is componentwise non-decreasing.}
\end{equation} 
But (\ref{eq:InverseFz}) holds as soon as, for every $x\in\R$, $\mathbf{z}\longmapsto F_{\mathbf{z}}(x)$ is non-increasing with respect to each argument. Indeed, for every $\mathbf{z}=(z_2,\cdots,z_n)$, $\mathbf{z}'=(z'_2,\cdots,z'_n)$ in $\R^{n-1}$ such that $\mathbf{z}\leq \mathbf{z}'$ (i.e. $z_i\leq z'_i$ for every $i=2,\cdots,n$), and for every $u\in[0,1]$,  
\begin{align*}
F_{\mathbf{z}'}\left(F_{\mathbf{z}}^{-1}(u)\right)&\leq F_{\mathbf{z}}\left(F_{\mathbf{z}}^{-1}(u)\right)\quad\text{(since }\mathbf{z}\mapsto F_{\mathbf{z}}(x)
\text{ is componentwise non-increasing})\\
&=u=F_{\mathbf{z}'}\left(F_{\mathbf{z}'}^{-1}(u)\right).
\end{align*}
Since  $p$ is MTP$_2$, then, for every $\mathbf{z}\leq \mathbf{z}'$,
$$
p(y,\mathbf{z}')p(x,\mathbf{z})\geq
p(y,\mathbf{z})p(x,\mathbf{z}')\quad\text{if }y\geq x,
$$ 
and
$$
p(y,\mathbf{z})p(x,\mathbf{z}')\geq
p(y,\mathbf{z}')p(x,\mathbf{z})\quad\text{if }y\leq x.
$$ 
Then, for every $\mathbf{z}\leq \mathbf{z}'$ in $\R^{n-1}$,
\begin{align*}
\frac{1}{F_{\mathbf{z}}(x)}&=1+\frac{\displaystyle\int^{\infty}_{x}f(y)p(y,\mathbf{z})dy}{\displaystyle\int_{-\infty}^{x}f(y)p(y,\mathbf{z})dy}
=1+\frac{\displaystyle\int^{\infty}_{x}f(y)p(y,\mathbf{z})p(x,\mathbf{z}')dy}{\displaystyle\int_{-\infty}^{x}f(y)p(y,\mathbf{z})p(x,\mathbf{z}')dy}\\
&\leq 1+\frac{\displaystyle\int^{\infty}_{x}f(y)p(y,\mathbf{z}')p(x,\mathbf{z})dy}{\displaystyle\int_{-\infty}^{x}f(y)p(y,\mathbf{z}')p(x,\mathbf{z})dy}
=1+\frac{\displaystyle\int^{\infty}_{x}f(y)p(y,\mathbf{z}')dy}{\displaystyle\int_{-\infty}^{x}f(y)p(y,\mathbf{z}')dy}:=\frac{1}{F_{\mathbf{z}'}(x)};
\end{align*}
which proves that, for every $x\in\R$, $\mathbf{z}\longmapsto F_{\mathbf{z}}(x)$ is componentwise non-increasing, and then (\ref{eq:InverseFz}) holds.
\item[$\bullet$]\textbf{Case $\mathbf{l>2}$. }Suppose that (E$_{l-1}$) holds for every function in $\mathcal{I}_{l-1}$. Let us prove (E$_l$) for every fixed $\phi\in\mathcal{I}_l$.\\
For every $\mathbf{z}=(z_l,\cdots,z_n)\in\R^{n-l+1}$, we define:
\begin{align*}
\mathcal{K}(l,\mathbf{z})&=\frac{\E\left[\left.\phi(X_1,\cdots,X_l)\prod\limits_{k=1}^lf_k(X_k)\right|X_l=z_l,\cdots,X_n=z_n\right]}{\E\left[\left. \prod\limits_{k=1}^lf_k(X_k)\right|X_l=z_l,\cdots,X_n=z_n\right]}\\
&=\frac{\displaystyle\int_{-\infty}^{\infty}\left(\displaystyle\int_{\R^{l-2}} \phi(\overline{x},x_{l-1},z_l)\prod\limits_{k=1}^{l-1}f_k(x_k)p(\overline{x},x_{l-1},\mathbf{z})d\overline{x}\right)dx_{l-1}}{\displaystyle\int_{-\infty}^{\infty}\left(\displaystyle\int_{\R^{l-2}}  \prod\limits_{k=1}^{l-1}f_k(x_k)p(\overline{x},x_{l-1},\mathbf{z})d\overline{x}\right)dx_{l-1}},
\end{align*}
where $\overline{x}=(x_1,\cdots,x_{l-2})$ and $d\overline{x}=dx_1\cdots dx_{l-2}$.\\
Now, set $\mathbf{X}=(X_l,\cdots,X_n)$, and consider the functions $\widehat{\phi}:\R^{n-l+2}\to\R$ and $\widehat{p}:\R^{n-l+2}\to\R_+^{\ast}$ defined respectively by:
\begin{align*}
\widehat{\phi}(u,\mathbf{z})&=\frac{\displaystyle\int_{\R^{l-2}} \phi(\overline{x},u,z_l)\prod\limits_{k=1}^{l-2}f_k(x_k)p(\overline{x},u,\mathbf{z})d\overline{x}}{\displaystyle\int_{\R^{l-2}}  \prod\limits_{k=1}^{l-2}f_k(x_k)p(\overline{x},u,\mathbf{z})d\overline{x}}\\
&=\frac{\E\left[\left.\phi(X_1,\cdots,X_{l-1},z_l)\prod\limits_{k=1}^{l-1}f_k(X_k)\right|X_{l-1}=u,\mathbf{X}=\mathbf{z}\right]}{\E\left[\left. \prod\limits_{k=1}^{l-1}f_k(X_k)\right|X_{l-1}=u,\mathbf{X}=\mathbf{z}\right]},
\end{align*} 
and
$$
\widehat{p}(u,\mathbf{z})=\int_{\R^{l-2}}\prod\limits_{k=1}^{l-2}f_k(x_k)p(\overline{x},u,\mathbf{z})d\overline{x}.
$$
By the induction hypothesis (E$_{l-1}$), $\widehat{\phi}$ belongs to $\mathcal{I}_{n-l+2}$, and, using Point 2) of Theorem \ref{theo:MTP2Density}, $\widehat{p}$ satisfies (\ref{eq:MTP2}). Moreover,
$$
\mathcal{K}(l,\mathbf{z})=\frac{\displaystyle\int_{-\infty}^{\infty}\widehat{\phi}(x_{l-1},\mathbf{z})f_{l-1}(x_{l-1})\widehat{p}(x_{l-1},\mathbf{z})dx_{l-1}}{\displaystyle\int_{-\infty}^{\infty} f_{l-1}(x_{l-1})\widehat{p}(x_{l-1},\mathbf{z})dx_{l-1}}.
$$
Since $\widehat{p}$ is MTP$_2$, then, using the same computations as in the Case $l=2$, we show that, for every $y\in\R$,
$$
\mathbf{z}\longmapsto\frac{\displaystyle\int_y^{\infty}f_{l-1}(x_{l-1})\widehat{p}(x_{l-1},\mathbf{z})dx_{l-1}}{\displaystyle\int_{-\infty}^yf_{l-1}(x_{l-1})\widehat{p}(x_{l-1},\mathbf{z})dx_{l-1}}
\,\text{}
$$
is non-decreasing with respect to each argument; which yields that $\mathbf{z}\longmapsto\mathcal{K}(l,\mathbf{z})$ is also non-decreasing with respect to each argument.
\end{proof}
\begin{rem}\label{rem:SCMequivMTP2}
We failed to find a SCM process which is not MTP$_2$ due to that SCM and MTP$_2$ properties coincide for several processes. For example, let $(X_1,X_2)$ be a SCM random vector which has the law $\Pb(X_1=i,X_2=j)=p(i,j)$ $(i,j\in\Z)$. Then, $(X_1,X_2)$ is MTP$_2$. To prove this, we may assume without loss of generality that $p$ is strictly positive. Since $(X_1,X_2)$ is SCM, then, for every bounded and strictly positive function $f$, and for every $a\in\Z$,
\[
n\longmapsto\frac{\displaystyle\sum\limits_{k=a}^{+\infty}f(k)p(k,n)}{\displaystyle\sum\limits_{k=-\infty}^{a-1}f(k)p(k,n)}
\,\text{ is non-decreasing,}
\]
i.e. for every $n\leq n'$, every $a\in\Z$ and every bounded and strictly positive function $f:\Z\to\R_+$,
\begin{equation*}
\sum\limits_{k'=a}^{+\infty}f(k')\left(\sum\limits_{k=-\infty}^{a-1}f(k)\,[p(k,n)p(k',n')-p(k,n')p(k',n)]\right)\geq0
\end{equation*} 
which is equivalent to
\begin{equation}\label{eq:SCMequivTP2}
 \sum\limits_{k'=a}^{+\infty}g(k')\left(\sum\limits_{k=-\infty}^{a-1}h(k)\,[p(k,n)p(k',n')-p(k,n')p(k',n)]\right)\geq0
\end{equation} 
for every bounded and strictly positive functions $g:\llbracket a,+\infty\llbracket\to\R_+$ and $h:\rrbracket-\infty,a-1\rrbracket\to\R_+$. But (\ref{eq:SCMequivTP2}) implies
\[
\left.
\begin{array}{c}
\forall\,k'\in\llbracket a,+\infty\llbracket,\,\forall\,h:\rrbracket-\infty,a-1\rrbracket\to\R_+,\\ \\
\qquad\qquad\qquad\sum\limits_{k=-\infty}^{a-1}h(k)\,[p(k,n)p(k',n')-p(k,n')p(k',n)]\geq0
\end{array}
\right\} 
\]
which in turn implies
\[
\forall\,k'\in\llbracket a,+\infty\llbracket,\forall\,k\in\rrbracket-\infty,a-1\rrbracket,\,  p(k,n)p(k',n')-p(k,n')p(k',n)\geq0.
\]
Since $a$, $n$ and $n'$ are arbitrary integers, we deduce that $p$ is TP$_2$. As a consequence, every SCM process taking values in a discrete subset of $\R$ has TP$_2$ bidimensional marginals. In particular, a Markov process with discrete state space is SCM if and only if it is MTP$_2$.
\end{rem}
Here is a direct consequence of Theorem \ref{theo:SCMrvMTP2}.
\begin{corol}\label{corol:SCMMarkovTP2}
Every stochastic process with MTP$_2$ finite-dimensional  marginals is SCM. In particular, if $X:=(X_t,t\geq0)$ is a real valued Markov process such that $X$ has an absolutely continuous and TP$_2$ transition kernel, then $X$ is SCM.
\end{corol} 
\begin{rem}
By Corollary \ref{corol:SCMMarkovTP2},  the processes below are SCM:
\begin{enumerate}
\item[i)]processes with independent and log-concave increments,
\item[ii)]absolute values of processes with independent, symmetric and PF$_{\infty}$ increments,
\item[iii)]one-dimensional diffusions with absolutely continuous transition kernel.
\item[iv)] Gaussian random vectors with an invertible covariance matrix such that the inverse matrix has negative off-diagonal elements.
\end{enumerate}  
\end{rem}

\section{Applications of Strong conditional monotonicity to peacocks}
We use strong conditional monotonicity results to study some generalisations of the Carr-Ewald-Xiao theorem (see (\ref{eq:CEX})). 

\subsection{Peacocks obtained by integrating with respect to a finite positive measure}
The following result was proved in \cite{BPR1}.
\begin{theorem}\label{theo:CM}
Let $(X_{\lambda},\lambda\geq0)$ be a real valued right-continuous process which is conditionally monotone in the sense that, for every $n\in\N^{\ast}$, every $i\in\{1,\cdots,n\}$, every $0\leq\lambda_1<\cdots<\lambda_n$, and every $\phi\in\mathcal{I}_n$,
\begin{equation}\label{eq:CM}\tag{CM}
z\longmapsto\E[\phi(X_{\lambda_1},\cdots,X_{\lambda_n})|X_{\lambda_i}=z]\,\text{is non-decreasing.}
\end{equation}
Suppose that, for every $t\geq0$ and every compact $K\subset\R_+$,
\begin{equation}\label{eq:CMInt}
\E\left[\exp\left(t\sup\limits_{\lambda\in K}X_{\lambda}\right)\right]<\infty\,\text{ and }\,
\inf\limits_{\lambda\in K}\E[\exp(tX_{\lambda})]>0.
\end{equation}
Then, for every finite positive measure $\mu$ on $\R_+$,
$$
\left(A^{(\mu)}_t:=\int_0^{\infty}\frac{e^{tX_{\lambda}}}{\E\left[e^{tX_{\lambda}}\right]}\mu(d\lambda),t\geq0\right)\text{ is a peacock.}
$$
\end{theorem}
This result is a generalisation of Carr-Ewald-Xiao theorem. Indeed, by making the change of variable $s=t\lambda$ in (\ref{eq:CEX}), the Brownian scaling property yields:
\begin{equation}\label{eq:CEXchangeVa}
\forall\,t\geq0,\,\text{ }{\bf N}_t=\int_0^1e^{B_{t\lambda}-\frac{t\lambda}{2}}d\lambda 
 \=\int_0^1 e^{\sqrt{t}B_{\lambda}-\frac{t\lambda}{2}}d\lambda.
\end{equation} 
Observe that $(B_{\lambda},\lambda\geq0)$ is a right-continuous conditionally monotone process since it is a Lï¿½vy process with log-concave increments. Then, by Theorem \ref{theo:CM}, 
$$
\left(\mathbf{A}_t:=\int_0^{1}\frac{e^{\sqrt{t}B_{\lambda}}}{\E\left[e^{\sqrt{t}B_{\lambda}}\right]}\,d\lambda,t\geq0\right)\text{ is a  peacock,}
$$
and we recover (\ref{eq:CEX}) thanks to (\ref{eq:CEXchangeVa}).\\
 In \cite{BPR1}, further examples of conditionally monotone processes are presented. For example, the Gamma subordinator and ``well-reversible" diffusions at a fixed time are conditionally monotone. We refer to \cite{BPR1} for the definition and some properties of ``well-reversible" diffusions. Moreover, Theorem \ref{theo:CM} applies to stochastic processes with MTP$_2$ finite-dimensional marginals (such as one-dimensional diffusions) since they satisfy (\ref{eq:SCM}) (which implies (\ref{eq:CM})). Note that one-dimensional diffusions are not necessarily ``well-reversible" at a fixed time. Indeed, ``well-reversible" diffusions at fixed time are unique strong solutions of  stochastic differential equations.
 
\subsection{Peacocks obtained by normalisation.}
Let $(V_t,t\geq0)$ be an integrable real valued process with a strictly positive mean, i.e. $\E[|V_t|]<\infty$ and $\E[V_t]>0$. Consider the process
$$
\left(N_t:=\frac{V_t}{\E[V_t]},t\geq0\right).
$$
Observe that $\E[N_t]=1$ for every $t\geq0$. Since $\E[N_t]$ does not depend on $t$ (which is a necessary condition to be a peacock), it is natural to look for processes $(V_t,t\geq0)$ for which $(N_t,t\geq0)$ is a peacock. Many examples of such processes are presented in \cite{BPR2}. Note that if ${\bf V}_t:=\displaystyle\int_0^te^{B_s-\frac{s}{2}}ds$, then (\ref{eq:CEX}) is equivalent to:
$$
\left({\bf N}_t:=\frac{{\bf V}_t}{\E[{\bf V}_t]},t\geq0\right)\text{ is a peacock.}
$$
One may also investigate on processes $(V_t,t\geq0)$ such that the centered process 
$(C_t:=V_t-\E[V_t],t\geq0)$ 
is a peacock. We do not treat this case here and refer to \cite{BPR2} and \cite{HPRY} for main results.\\
We deal with processes of the forms
\begin{equation}\label{eq:SCMFormVt1}\tag{F$_1$}
V^1_t=\exp\left(\int_0^tq(\lambda,X_{\lambda})\mu(d\lambda)\right)
\end{equation}
and
\begin{equation}\label{eq:SCMFormVt2}\tag{F$_2$}
V^2_t=\exp\left(\int_0^{\infty}q(\lambda,tX_{\lambda})\mu(d\lambda)\right),
\end{equation} 
where $\mu$ is a positive Radon measure, $q:\R_+\times \R\to\R$ a continuous function such that, for every $s\geq0$, $q_s:x\longmapsto q(s,x)$ is  non-decreasing (resp. non-increasing), and where $(X_{\lambda},\lambda\geq0)$ is a real valued process. Our purpose is to answer the following question: 
\begin{equation}\label{eq:Question}\tag{Q}
\text{Under which conditions is }\left(N^i_t:=\frac{V^i_t}{\E[V^i_t]},t\geq0\right)\,(i=1,2)
\text{ a peacock?}
\end{equation}
Remark that, in (\ref{eq:SCMFormVt1}), the parameter $t$ is a time parameter, while in (\ref{eq:SCMFormVt2}), it is a dilatation parameter. For this reason, we call {\em peacocks with respect to maturity} peacocks of type $N^1$ and {\em peacocks with respect to volatility} peacocks of type $N^2$.\\
For processes of type $N^1$, a partial answer to (\ref{eq:Question}) is given in \cite{BPR2} when  $X$ has  independent and log-concave increments. We extend this result to real valued processes  which satisfy (\ref{eq:SCM}). In particular, we  prove under some integrability hypotheses that  processes with  MTP$_2$ finite-dimensional marginals solve (\ref{eq:Question}).  

\subsubsection{Peacocks with respect to maturity} 
\begin{theorem}\label{theo:SCMPcocNorme}
Let $(X_{\lambda},\lambda\geq0)$ be a right-continuous process which satisfies (\ref{eq:SCM}). Let $\mu$ be a positive Radon measure and let $q:\R_+\times \R\to\R$ be a continuous function such that, for every $t\geq0$:
\begin{enumerate}
\item[i)]$y\longmapsto q(t,y)$ is non-decreasing (resp. non-increasing),
\item[ii)]the following integrability properties hold:
\begin{equation}\label{eq:SCMmarkovTpINT1}
\Theta_t:=\exp\left(\mu([0,t])\sup_{0\leq s\leq t}q(s,X_s)\right)\text{ is }\text{ integrable}
\end{equation}
and
\begin{equation}\label{eq:SCMmarkovTpINT2}
\Delta_t:=\E\left[\exp\left(\mu([0,t])\inf_{0\leq s\leq t}q(s,X_s)\right)\right]>0.
\end{equation}
\end{enumerate}
Then, 
\begin{equation}\label{eq:SCMpcoc}
\left(N_t:=\frac{\exp\left(\displaystyle\int_0^tq(s,X_s)\,\mu(ds)\right)}{\E\left[\exp\left(\displaystyle\int_0^tq(s,X_s)\,\mu(ds)\right)\right]},t\geq0\right)
\quad\text{is a peacock.}
\end{equation}
\end{theorem}
\begin{proof}
We only consider the case where $y\longmapsto q(\lambda,y)$ is non-decreasing.\\
Let $T>0$ be fixed.  
\item[\textbf{1)}] We first suppose that $\mu$ has the form:
\begin{equation}\label{eq:CMeForme}
1_{[0,T]}d\mu=\sum_{i=1}^ra_i\delta_{\lambda_i}, 
\end{equation}
where  $r\in\N$, $r\geq2$, $a_1\geq0,a_2\geq0,\dots,a_r\geq0$,
$\sum_{i=1}^ra_i=\mu([0,T])$, 
$0\leq\lambda_1<\lambda_2<\dots<\lambda_r\leq T$, and where $\delta_{\lambda_i}$ is the Dirac measure at point $\lambda_i$.\\
We show that,  
\[
\left(N_n:=\exp\left(\sum_{i=1}^na_iq(\lambda_i,X_{\lambda_i})-h(n)\right),
n\in\{1,2,\cdots,r\}\right)
\text{ is a peacock,}
\]
where
\[
h(n):=\log\E\left[\exp\left(\sum_{i=1}^na_iq(\lambda_i,X_{\lambda_i})\right)\right].
\]
Note that:
\[
\E[N_n-N_{n-1}]=0, \text{ for every } n\in\{1,2,\cdots,r\}
\]
with
\begin{align*}
N_n-N_{n-1}&=N_{n-1}\left(e^{a_nq(\lambda_n,X_{\lambda_n})-h(n)+h(n-1)}-1\right)
=N_{n-1}\left(e^{\widetilde{q}_n(X_{\lambda_n})}-1\right)
\end{align*}
and
\[
\widetilde{q}_n(y)=a_nq(\lambda_n,y)-h(n)+h(n-1).
\]
Then, for every convex function $\psi\in\mathbf{C}$, 
\begin{align*}
\E[\psi(N_n)]-\E[\psi(N_{n-1})]  
&\geq\E\left[\psi'(N_{n-1})N_{n-1}\left(e^{\widetilde{q}_n(X_{\lambda_n})}-1\right)\right]\\
&=\E\left[K(n,X_{\lambda_n})\E[N_{n-1}|X_{\lambda_n}]
\left(e^{\widetilde{q}_n(X_{\lambda_n})}-1\right)\right],
\end{align*}
where
\begin{eqnarray*}
K(n,z)=\dfrac{\E[\psi'(N_{n-1})N_{n-1}|X_{\lambda_n}=z]}{\E[N_{n-1}|X_{\lambda_n}=z]}.
\end{eqnarray*}
Observe that the function $\phi:\R^{n-1}\to\R_+$  given by:
\[
\phi(x_1,\dots,x_{n-1})=\psi'\left[\exp\left(\sum_{i=1}^{n-1}a_iq(\lambda_i,x_i)-h(n-1)\right)\right]
\]
belongs to $\mathcal{I}_{n-1}$.  
If, for every $i\in\N^*$, we define:
\[
f_i(x)=e^{a_iq(\lambda_i,x)},\text{ for every }x\in\R;
\]
then, for every $n\in\{2,\cdots,r\}$,
\[
N_{n-1}=e^{-h(n-1)}\prod_{k=1}^{n-1}f_k(X_{\lambda_k})
\]
and
\[
K(n,z)=\dfrac{\E\left[\left.\displaystyle \phi(X_{\lambda_1},\dots,X_{\lambda_{n-1}})\prod\limits_{k=1}^{n-1}f_k(X_{\lambda_k})
\right|X_{\lambda_n}=z\right]}
{\E\left[\left.\displaystyle\prod\limits_{k=1}^{n-1}f_k(X_{\lambda_k})\right|X_{\lambda_n}=z\right]}.
\]
Note that $K(n,z)$ is well defined since, for every $n\in\{1,2,\cdots,r\}$,
\begin{eqnarray*}
\E\left[\prod\limits_{k=1}^{n}f_k(X_{\lambda_k})\right]&=&
\E\left[\exp\left(\sum\limits_{k=1}^na_iq(\lambda_i,X_{\lambda_i})\right)\right]\\
&\leq&\E\left[\exp\left(\sup\limits_{0\leq\lambda\leq T}q(\lambda,X_{\lambda})
\sum\limits_{k=1}^na_i\right)\right]\\
&\leq&\E\left[\exp\left(\sup\limits_{0\leq\lambda\leq T}q(\lambda,X_{\lambda})
\sum\limits_{k=1}^ra_i\right)\vee1\right]\\
&=&\E\left[\exp\left(\alpha(T)\sup\limits_{0\leq\lambda\leq T}q(\lambda,X_{\lambda})
\right)\vee1\right]\\ &=&\E[\Theta_T\vee1]<\infty.
\end{eqnarray*}
By (\ref{eq:SCM}),  $K(n,z)$ is non-decreasing with respect to $z$.\\
Now, for every  $n\in\N^{\ast}$, we denote by $(\widetilde{q}_{n})^{-1}$ the right-continuous inverse of $\widetilde{q}_n$ and we set:
\[
V(n,X_{\lambda_n}):=K(n,X_{\lambda_n})\E[N_{n-1}|X_{\lambda_n}]
\left(e^{\widetilde{q}_n(X_{\lambda_n})}-1\right).
\]
Then,
 
\begin{enumerate}
\item[i)] if $X_{\lambda_n}\leq (\widetilde{q}_{n})^{-1}(0)$, then $e^{\widetilde{q}_n(X_{\lambda_n})}-1\leq0$ and 
\begin{align*}
V(n,X_{\lambda_n})\geq  
K\left(n, (\widetilde{q}_{n})^{-1}(0)\right)
\E[N_{n-1}|X_{\lambda_n}]
\left(e^{\widetilde{q}_n(X_{\lambda_n})}-1\right),
\end{align*}
\item[ii)]  if $X_{\lambda_n}\geq (\widetilde{q}_{n})^{-1}(0)$, then  $e^{\widetilde{q}_n(X_{\lambda_n})}-1\geq0$ and 
\begin{align*}
V(n,X_{\lambda_n})\geq  
K\left(n, (\widetilde{q}_{n})^{-1}(0)\right)
\E[N_{n-1}|X_{\lambda_n}]
\left(e^{\widetilde{q}_n(X_{\lambda_n})}-1\right).
\end{align*}
\end{enumerate}
As a consequence,
\begin{align*}
&\E[\psi(N_n)]-\E_x[\psi(N_{n-1})]\\
&\geq \E[ V(n,X_{\lambda_n})] 
\geq K\left(n, (\widetilde{q}_{n})^{-1}(0)\right)
\E\left[\E[N_{n-1}|X_{\lambda_n}]
\left(e^{\widetilde{q}_{n}(X_{\lambda_n})}-1\right)\right]\\
&=K\left(n, (\widetilde{q}_{n})^{-1}(0)\right) 
\E\left[N_{n-1} 
\left(e^{\widetilde{q}_n(X_{\lambda_n})}-1\right)\right]\\
&=K\left(n, (\widetilde{q}_{n})^{-1}(0)\right) 
\E\left[N_n-N_{n-1}\right]=0; 
\end{align*}
which shows that, for every  integer $r\geq2$, 
\[
\left(N_n:=\exp\left(\sum_{i=1}^na_iq(\lambda_i,X_{\lambda_i})-h(n)\right),
n\in\{1,2,\cdots,r\}\right)
\text{ is a peacock.}
\]
\vspace{0.0cm}\\
\textbf{2) }We consider $\nu=1_{[0,T]}d\mu$, and, for every $0\leq t\leq T$, we set:
\[
N^{(\nu)}_t=\dfrac{\exp\left(\displaystyle\int_0^t q(u,X_u)\nu(du)\right)}
{\E\left[\exp\left(\displaystyle\int_0^t q(u,X_u)\nu(du)\right) \right]}.
\]
Since the function $\lambda\in[0,T]\longmapsto q(\lambda,X_\lambda)$
is right-continuous and bounded from above by $\sup\limits_{0\leq \lambda\leq T}|q(\lambda,X_\lambda)|$ which is finite a.s., there exists a sequence $(\nu_n,n\in\N)$ of measures of the form (\ref{eq:CMeForme}), such that,
for every $n\in\N$, supp$\,\nu_n\subset[0,T]$, $\int\nu_n(du)=\int\nu(du)$ and, for every $0\leq t\leq T$,
\begin{equation}\label{eq:cvge1}
\lim\limits_{n\to\infty} \exp\left(\int_0^t q(u,X_u)\nu_n(du)\right)=
\exp\left(\int_0^t q(u,X_u)\nu(du)\right)\text{ a.s.}
\end{equation}
Moreover, for every $0\leq t\leq T$ and every $n\in\N$,
\begin{align*}
& \exp\left(\displaystyle\int_0^t q(u,X_u)\nu_n(du)\right)  \\
&\leq \exp\left(\sup\limits_{0\leq\lambda\leq T}q(\lambda,X_\lambda)\displaystyle\int_0^t\nu_n(du)\right) \\
&\leq \exp\left(\sup\limits_{0\leq\lambda\leq T} q(\lambda,X_\lambda)\displaystyle\int_0^T\nu_n(du)\right)\vee1 \\
&= \exp\left(\sup\limits_{0\leq\lambda\leq T}q(\lambda,X_\lambda)\displaystyle\int_0^T\nu(du)\right)\vee1=
 \Theta_T\vee1 
\end{align*}
which is integrable from (\ref{eq:SCMmarkovTpINT1}). 
By the dominated convergence theorem,
\begin{equation}\label{eq:cvge2}
\lim\limits_{n\to\infty}\E\left[ \exp\left(\displaystyle\int_0^t q(u,X_u)\nu_n(du)\right)\right]=
\E\left[\exp\left(\displaystyle\int_0^t q(u,X_u)\nu(du)\right)\right].
\end{equation}
Using (\ref{eq:cvge1}) and (\ref{eq:cvge2}), we obtain:
\begin{equation}\label{eq:cvge3}
\lim\limits_{n\to\infty}N_t^{(\nu_n)}=N_t^{(\nu)}\text{ a.s., for every }0\leq t\leq T.
\end{equation}
But, we proved in $\textbf{1)}$ that:
\begin{equation}
\left(N_t^{(\nu_n)},0\leq t\leq T\right)\text{ is a peacock for every }n\in\N,
\end{equation}
i.e., for every $0\leq s<t\leq T$ and every $\psi\in\mathbf{C}$:
\begin{equation}\label{eq:pcoc-d}
\E\left[\psi(N_s^{(\nu_n)})\right]\leq\E_x\left[\psi(N_t^{(\nu_n)})\right].
\end{equation}
Besides,
\begin{equation}\label{eq:int3}
 \sup\limits_{0\leq t\leq T}\sup\limits_{n\geq0}\left|N_t^{(\nu_n)}\right| 
\leq  \dfrac{\Theta_T\vee1}{\Delta_T\wedge1},
\end{equation}
which is integrable from  (\ref{eq:SCMmarkovTpINT1}) and (\ref{eq:SCMmarkovTpINT2}). Using the dominated convergence theorem, we pass to the limit in  (\ref{eq:pcoc-d}) as $n\to\infty$ and deduce that 
$(N_t^{(\mu)},0\leq t\leq T)$ is a peacock for every $T>0$.
\end{proof}

Now, we prove a version of Theorem \ref{theo:SCMPcocNorme} for a squared Bessel of dimension $0$ (denoted BESQ$^0$). Note that the transition function of a BESQ$^0$ is not absolutely continuous with respect to Lebesgue measure. Then Theorem \ref{theo:SCMrvMTP2} does not apply. In particular,  the finite-dimensional marginals of a BESQ$^0$ do not satisfy (\ref{eq:SCM}). Nevertheless, a limit theorem due to Feller \cite{Fe} for critical Galton-Watson branching processes allows to exhibit peacocks of type (\ref{eq:SCMpcoc}).  
 
\begin{exa}\label{exa:Galton-Watson}(A version of Theorem \ref{theo:SCMPcocNorme} for a BESQ$^0$).\\
For every $k\in\N^{\ast}$, let $Z^k:=\left(Z_n^k,n\in\N\right)$ denote a Galton-Watson branching process starting with $k$ individuals, and which has a geometric reproduction law $\nu$ of parameter $\dfrac{1}{2}$, i.e. 
$$
\nu(i)=2^{-i-1},\quad\text{for every }i\in\N.
$$
For every $k\in \N^{\ast}$, $Z^k$ is an homogeneous Markov chain with values in $\N$, and its transition probability matrix $Q$ is given by:
$$
\forall\,j\in\N,\quad Q(0,j)=
\left\{
\begin{array}{ll}
1&\text{if }j=0\\
0&\text{otherwise}
\end{array}
\right.
$$
and
$$
\forall\,(i,j)\in\N^{\ast}\times\N,\quad
Q(i,j)=
\left(
\begin{array}{cc}
i+j-1\\j
\end{array}
\right)2^{-(i+j)}.
$$
We consider the family  $\left(Q^{(n)},n\in\N\right)$ of transition functions defined on $\N\times\N$ by:
$$
Q^{(0)}(i,j)=\left\{
\begin{array}{ll}
1&\text{if }i=j\\
0&\text{otherwise}
\end{array}
\right.
$$
and
$$
\forall\,n\geq1,\quad Q^{(n+1)}(i,j)=\sum\limits_{n\in\N}Q(i,m)Q^{(n)}(m,j).
$$
Since the function 
$(i,j)\longmapsto\left(
\begin{array}{cc}
i+j-1\\j
\end{array}
\right)$ is $\text{TP}_2$, we deduce from (\ref{eq:ConvProduct}) that $Q^{(n)}$ is $\text{TP}_2$ for every $n\in\N$.\\
For every $\lambda\geq0$ and every $k\in\N^{\ast}$, we set:
$$
Y^k_{\lambda}=\frac{1}{k}Z^k_{[k\lambda]},
$$
where $[\cdot]$ denotes the floor function. Then, $\left(Y^k_{\lambda},\lambda\geq0\right)$ is a Markov process with values in $\dfrac{1}{k}\N$, and its transition function $(P_{\zeta,\eta},0\leq \zeta<\eta)$ is given by:
$$
\forall\,x,y\in\frac{1}{k}\N,\quad P_{\zeta,\eta}(x,y)=Q^{([k\eta]-[k\zeta])}(kx,ky).
$$
Observe that $P_{\zeta,\eta}$ is $\text{TP}_2$ for every $\zeta$ and $\eta$. By Theorem \ref{theo:SCMrvMTP2}, for every $k\in\N^{\ast}$, $\left(Y^k_{\lambda},\lambda\geq0\right)$ is SCM.\\ 
Let $q:\R_+\times\R\to\R$ be bounded, continuous and such that, for every $\lambda\geq0$, $x\longmapsto q(\lambda,x)$ is non-decreasing (resp. non-increasing). It follows from Theorem \ref{theo:SCMPcocNorme} that, for every $a_i\geq0$, $i\in\N^{\ast}$, and for every strictly increasing sequence $(\lambda_i,i\geq1)$ in $\R_+^{\ast}$,  
\begin{equation}\label{eq:GWpMarkovTP2}
\left(N^{k}_n:=\frac{\exp\left(\displaystyle\sum\limits_{i=1}^na_iq(\lambda_i,Y^{k}_{\lambda_i})\right)}{\E\left[\exp\left(\displaystyle\sum\limits_{i=1}^na_iq(\lambda_i,Y^{k}_{\lambda_i})\right) \right]},n\in\N\right)\quad\text{is a peacock.}
\end{equation}
A result due to Feller \cite{Fe} states that, as $k$ tends to $\infty$, $\left(Y^k_{\lambda},\lambda\geq0\right)$ converges in distribution to 
$\left(Y^{\infty}_{\lambda},\lambda\geq0\right)$, which is the unique strong solution of:   
$$
dZ_{\lambda}=\sqrt{2Z_{\lambda}}dB_{\lambda},\quad Z_0=1,
$$
where $(B_{\lambda},\lambda\geq0)$  is a standard Brownian motion. In particular, $\left(Y^k_{\lambda},\lambda\geq0\right)$ converges in sense of finite distributions to $\left(Y^{\infty}_{\lambda},\lambda\geq0\right)$. Then, (\ref{eq:GWpMarkovTP2}) yields:
\begin{equation}\label{eq:BESQ0discretMarkovTP2}
\left(N^{\infty}_n:=\frac{\exp\left(\displaystyle\sum\limits_{i=1}^na_i q(\lambda_i,Y^{\infty}_{\lambda_i})\right)}{\E\left[\exp\left(\displaystyle\sum\limits_{i=1}^na_iq(\lambda_i,Y^{\infty}_{\lambda_i})\right) \right]},n\in\N\right)\quad\text{is a peacock.}
\end{equation}
As a consequense, we obtain the following result:
\begin{corol}
Let $(Y_t,t\geq0)$ be a BESQ$^0$ issued from 1, and let $q:\R_+\times\R_+\to\R$ be a continuous and bounded function such that, for every $\lambda\geq0$,  $y\longmapsto q(\lambda,y)$ is non-decreasing (resp. non-increasing). Then, for every positive Radon measure $\mu$ on $\R_+$,
$$
\left(N_t:=\frac{\exp\left(\displaystyle\int_0^tq(s,Y_s)\mu(ds)\right)}{\E\left[\exp\left(\displaystyle\int_0^tq(s,Y_s)\mu(ds)\right)\right]},t\geq0\right)
\quad\text{is a peacock.}
$$ 
\end{corol}
\end{exa}

\subsubsection{Peacocks with respect to volatility}
\begin{theorem}\label{theo:SCMvolatilePcocNorme}
Let $(X_{\lambda},\lambda\geq0)$ a right-continuous process which is SCM. Let $q:\R_+\times \R\to\R$ be a continuous function such that, for every $\lambda\geq0$, $x\longmapsto q(\lambda,x)$ is of $\mathcal{C}^1$ class, and let $\mu$ denote a positive Radon measure on $\R_+$ satisfying:
$$
\forall\,t\geq0,\,\text{ }\E\left[\exp\left(\int_0^{\infty}q(\lambda,tX_{\lambda})\mu(d\lambda)\right)\right]<\infty.
$$
We suppose that:
\begin{enumerate}
\item[i)]for every $\lambda\geq0$, the functions $x\longmapsto q(\lambda,x)$ and $x\longmapsto x\dfrac{\partial q}{\partial x}(\lambda,x)$ are non-decreasing (resp. non-increasing),
\item[ii)]for every $t,\lambda\geq0$, there exists $\alpha=\alpha(t,\lambda)>1$ such that:
\begin{equation}\label{eq:SCMF2Int0}
\E\left[|X_{\lambda}|^{\alpha}\left(\frac{\partial q}{\partial x}\right)^{\alpha}(\lambda,tX_{\lambda})\right]<\infty,
\end{equation}
\item[iii)]for every $t,\beta>0$, and for every compact $K\subset\R_+$,
\begin{equation}\label{eq:SCMF2Int1}
 \Theta^{(K)}_{t,\beta}:= \exp\left(\beta\,\sup\limits_{\lambda\in K} q(\lambda,tX_{\lambda})\right) 
 \quad\text{is integrable,}
\end{equation}
and  
\begin{equation}\label{eq:SCMF2Int2}
\Delta^{(K)}_{t,\beta}:=\E\left[\exp\left(\beta\,\inf\limits_{\lambda\in K} q(\lambda,t
X_{\lambda})\right)\right]>0.
\end{equation}
\end{enumerate}
Then,
\begin{equation}\label{eq:PcocVolatileMarkovTP2}
\left(N^{(\mu)}_t:=\frac{\exp\left(\displaystyle\int_0^{\infty}q(\lambda,tX_{\lambda})\mu(d\lambda)\right)}{ \E\left[\exp\left(\displaystyle\int_0^{\infty}q(\lambda,tX_{\lambda}) \mu(d\lambda)\right)\right]}, t\geq0\right)
\quad\text{is a peacock.}
\end{equation}
\end{theorem}
\begin{proof}
We shall suppose without loss of generality that, for every $\lambda\geq0$, the functions $x\longmapsto q(\lambda,x)$ and $x\longmapsto x\dfrac{\partial q}{\partial x}(\lambda,x)$ are non-decreasing.
\item[1)]We first treat the case where $\mu$ is of the form
$$
\mu=\sum\limits_{i=1}^na_i\delta_{\lambda_i},
$$
where $m\in\N^{\ast}$, $a_1\geq0,\cdots,a_n\geq0$, $0<\lambda_1<\cdots<\lambda_m$, and where $\delta_{\lambda}$ denote the Dirac measure at point $\lambda$. Precisely, we show that
$$
\left(N_t:=\exp\left(\sum\limits_{i=1}^ma_iq(\lambda_i,tX_{\lambda_i})-h(t)\right),t\geq0\right) \quad\text{is a peacock},
$$
with 
$$
h(t)=\log\E\left[\exp\left(\sum\limits_{i=1}^ma_iq(\lambda_i,tX_{\lambda_i})\right)\right].
$$
We set  $\overline{\mu}:=\sum\limits_{i=1}^ma_i$. Since the functions $x\longmapsto q(\lambda,x)$
and $x\longmapsto x\dfrac{\partial q}{\partial x}(\lambda,x)$ are non-decreasing, then, for every $0<b<c$, and every $t\in[b,c]$,
\begin{equation}\label{eq:SCMF2Sup01}
\exp\left(\sum\limits_{i=1}^ma_iq(\lambda_i,tX_{\lambda_i})\right)\leq\exp\left(\overline{\mu} \sup\limits_{i\in\{1,\cdots,m\}}q(\lambda_i,0)\right)+ \exp\left(\overline{\mu}\sup\limits_{i\in\{1,\cdots,m\}}q(\lambda_i,c\,X_{\lambda_i})\right),
\end{equation}
and for every $i\in\{1,\cdots,m\}$,
\begin{equation}\label{eq:SCMF2Sup02}
|X_{\lambda_i}|\frac{\partial q}{\partial x}(\lambda_i,tX_{\lambda_i})\leq\frac{c}{b}|X_{\lambda_i}| \frac{\partial q}{\partial x}(\lambda_i,c\,X_{\lambda_i}).
\end{equation}
We deduce from (\ref{eq:SCMF2Int0}), (\ref{eq:SCMF2Int1}), (\ref{eq:SCMF2Sup01}) and (\ref{eq:SCMF2Sup02}) that, for every $0<b<c$,
\begin{equation}\label{eq:SCMF2Sup}
\E\left[\sup\limits_{t\in[b,c]} \left\{\sum_{i=1}^ma_i|X_{\lambda_i}|\frac{\partial q}{\partial x}(\lambda_i,tX_{\lambda_i})\exp\left(\sum\limits_{k=1}^ma_k q(\lambda_k,tX_{\lambda_k})\right) \right\}\right]<\infty.
\end{equation}
Consequently, $h$ is continuous on $[0,+\infty[$, differentiable on $]0,+\infty[$, and for every $t>0$,
$$
h'(t)e^{h(t)}=\sum\limits_{i=1}^ma_i\,\E\left[X_{\lambda_i}\frac{\partial q}{\partial x}(\lambda_i,tX_{\lambda_i})\exp\left(\sum\limits_{k=1}^ma_k q(\lambda_k,tX_{\lambda_k})\right)\right],
$$ 
i.e.
\begin{equation}\label{eq:HmMtp01}
h'(t)=\sum\limits_{i=1}^ma_i\,\E\left[N_t\,X_{\lambda_i}\frac{\partial q}{\partial x}(\lambda_i,tX_{\lambda_i})\right].
\end{equation}
Now, define
\begin{equation}\label{eq:HmMtp02}
\widetilde{h}_{\lambda_i}(t)=\E\left[N_t\,X_{\lambda_i}\dfrac{\partial q}{\partial x}(\lambda_i,tX_{\lambda_i})\right]
\end{equation}
so that
\begin{equation}\label{eq:HmMtp03}
h'(t)=\sum\limits_{i=1}^ma_i\widetilde{h}_{\lambda_i}(t).
\end{equation}
Since $\E[N_t]=1$, then, for every $t>0$ and $i\in\{1,\cdots,n\}$, (\ref{eq:HmMtp02}) yields:
\begin{equation}\label{eq:HmMtpMean0}
 \E\left[N_t\left(X_{\lambda_i}\frac{\partial q}{\partial x}(\lambda_i,tX_{\lambda_i})-\widetilde{h}_{\lambda_i}(t)\right)\right]=0.
\end{equation}
On the other hand, if $\psi$ is a convex function  in $\mathbf{C}$, then (\ref{eq:SCMF2Int1}), (\ref{eq:HmMtp01}) and (\ref{eq:HmMtp03}) imply
$$
\frac{\partial}{\partial t}\E[\psi(N_t)]=\sum\limits_{i=1}^ma_i\,\E\left[\psi'(N_t)N_t\left(X_{\lambda_i}\frac{\partial q}{\partial x}(\lambda_i,tX_{\lambda_i})-\widetilde{h}_{\lambda_i}(t)\right)\right].
$$
Thus, it remains to prove that, for every $i\in\{1,\cdots,m\}$,
\begin{equation}\label{eq:HmMtpPcocDiscret}
\mathbf{\Delta}_i:=\E\left[\psi'(N_t)N_t\left(X_{\lambda_i}\frac{\partial q}{\partial x}(\lambda_i,tX_{\lambda_i})-\widetilde{h}_{\lambda_i}(t)\right)\right]\geq0.
\end{equation}
Observe that the function
$$
\phi:(x_1,\cdots,x_m)\longmapsto\psi'\left(\exp\left(\sum_{k=1}^ma_kq(\lambda_k,tx_k)-h(t)\right)\right)
$$ 
belongs to $\mathcal{I}_m$. Moreover, if, for every $k\in\{1,\cdots,m\}$, we set $f_k(x)=\exp(a_kq(\lambda_k,tx))$,
then
$$
N_t=e^{-h(t)}\prod_{k=1}^mf_k(X_{\lambda_k}).
$$
Therefore, by setting
$$
K_{i}(m,z):=\frac{\E\left[\left.\phi(X_{\lambda_1},\cdots,X_{\lambda_m})\prod\limits_{k=1}^m f_k(X_{\lambda_k})\right|X_{\lambda_i}=z\right]}{\E\left[\left.\prod\limits_{k=1}^m f_k(X_{\lambda_k})\right|X_{\lambda_i}=z\right]},
$$
for every $z\in \R$, and every $i\in\{1,\cdots,m\}$, we obtain
$$
\mathbf{\Delta}_i=\E\left[K_{i}(m,X_{\lambda_i})\E[N_t|X_{\lambda_i}] \left(X_{\lambda_i}\frac{\partial q}{\partial x}(\lambda_i,tX_{\lambda_i})-\widetilde{h}_{\lambda_i}(t)\right)\right].
$$
By hypothesis i), $\widetilde{q}_{\zeta_i}:\,x\longmapsto x\dfrac{\partial q}{\partial x}(\zeta_i,tx)-\widetilde{h}_{\zeta_i}(t)$ is continuous and non-decreasing; let $\widetilde{q}_{\zeta_i}^{-1}$ denote its right-continuous inverse. Since $(X_{\lambda},\lambda\geq0)$ is SCM, the function $z\longmapsto K_{i}(m,z)$ is non-decreasing, and we deduce from (\ref{eq:HmMtpMean0}) that:
\begin{align*}
\mathbf{\Delta}_i&\geq K_{i}\left(m,\widetilde{q}_{\lambda_i}^{-1}(0)\right)\E\left[\E[N_t|X_{\lambda_i}]\left(X_{\lambda_i}\frac{\partial q}{\partial x}(\lambda_i,tX_{\lambda_i})-\widetilde{h}_{\lambda_i}(t)\right)\right]\\
&=K_{i}\left(m,\widetilde{q}_{\lambda_i}^{-1}(0)\right)\E\left[N_t\left(X_{\lambda_i}\frac{\partial q}{\partial x}(\lambda_i,tX_{\lambda_i})-\widetilde{h}_{\lambda_i}(t)\right)\right]=0.
\end{align*} 
Thus, $(N_t,t\geq0)$ is a peacock.
\item[2)]If $\mu$ has a compact support contained in a compact interval of $\R_+$, then, following the same lines as Point 2) in the proof of Theorem \ref{theo:SCMPcocNorme}, we prove that $\left(N_t^{(\mu)},t\geq0\right)$ is a peacock.
\item[3)]In the general case, we consider the sequence $(\mu_n(d\lambda):=1_{[0,n]}\mu(d\lambda),n\in\N)$. Let $\psi\in\mathbf{C}$. By Point 2) above, $\left(N_t^{(\mu_n)},t\geq0\right)$ is a peacock for every $n$. Then 
\begin{equation}\label{eq:HmMtpPcocCompact}
\forall\,0\leq s\leq t,\quad\E\left[\psi(N_s^{(\mu_n)})\right]\leq\E\left[\psi(N_t^{(\mu_n)})\right].
\end{equation}
Moreover, it follows from Theorem \ref{theo:SCMPcocNorme} that, for every $t\geq0$,
$$
\left(N_t^{(\mu_n)}=\frac{\exp\left(\displaystyle\int_0^nq(\zeta,tX_{\zeta})\mu(d\zeta)\right)}{\E\left[\exp\left(\displaystyle\int_0^nq(\zeta,tX_{\zeta})\mu(d\zeta)\right) \right]},n\geq0\right)\quad\text{is a peacock,}
$$
in other terms, the sequence $\left(\E\left[\psi(N_t^{(\mu_n)})\right],n\geq0\right)$ is non-decreasing and bounded from above by $\E\left[\psi(N_t^{(\mu)})\right]$. Therefore, letting $n$ tends to $\infty$ in (\ref{eq:HmMtpPcocCompact}), we obtain:
$$
\forall\,0\leq s\leq t,\quad\E\left[\psi(N_s^{(\mu)})\right]\leq\E\left[\psi(N_t^{(\mu)})\right]
$$
which proves that $\left(N_t^{(\mu)},t\geq0\right)$ is a peacock.
\end{proof}
We end with some examples.
\begin{exa}
Let $\mu$ be a positive Radon measure. Let $(X_{\lambda},\lambda\geq0)$ be a right-continuous process having MTP$_2$ finite-dimensional marginals, and such that, for every $\beta>0$,
\begin{equation}\label{eq:HMexeMarkovTP2hm1}
\E\left[\exp\left(\beta\sup\limits_{0\leq \zeta\leq1}X_{\lambda}\right)\right]<\infty
\end{equation}
and
\begin{equation}\label{eq:HMexeMarkovTP2hm2}
\E\left[\exp\left(\beta\inf\limits_{0\leq \lambda\leq1}X_{\zeta}\right)\right]>0.
\end{equation}
We consider the function $q:[0,1]\times\R\to\R$ defined by:
$$
\forall\,(\lambda,x)\in[0,1]\times\R,\quad q(\lambda,x)=2x+\sqrt{1+\lambda+x^2}. 
$$
The following inequalities are immediate.
\begin{equation}\label{ineq:HMexeMarkovTP2}
\forall\,(\lambda,x)\in[0,1]\times\R,\quad e^{2x}\leq e^{q(\lambda,x)}<e^{2+3x}+e^{2+x}.
\end{equation}
Then, using (\ref{eq:HMexeMarkovTP2hm1}) and (\ref{ineq:HMexeMarkovTP2}), we have:
$$
\forall\,t\geq0,\quad\E\left[\exp\left(\mu([0,1])\sup\limits_{0\leq\lambda\leq1} q(\lambda,tX_{\lambda})\right)\right]<\infty.
$$
Moreover, (\ref{eq:HMexeMarkovTP2hm1}), (\ref{eq:HMexeMarkovTP2hm2}) and (\ref{ineq:HMexeMarkovTP2}) ensure that conditions (\ref{eq:SCMF2Int0})-(\ref{eq:SCMF2Int2}) of Theorem \ref{theo:SCMPcocNorme} are fulfilled.\\
On the other hand, $x\longmapsto q (\lambda,x)$ and $x\longmapsto x\dfrac{\partial q}{\partial x}(\lambda,x)$ are non-decreasing functions.  Therefore, by Theorem \ref{theo:SCMvolatilePcocNorme}, 
$$
\left(N_t:=\frac{\exp\left(\displaystyle\int_0^1q (\lambda,tX_{\lambda})\mu(d\lambda)\right)}{\E\left[\exp\left(\displaystyle\int_0^1q (\lambda,tX_{\lambda})\mu(d\lambda)\right)\right]},t\geq0\right)\quad\text{is a peacock.}
$$
\end{exa}
\begin{exa}
Let $X:=(X_u,u\geq0)$ denote a right-continuous process which has MTP$_2$ finite-dimensional marginals. We assume that $X$ enjoys the scale property of order $\gamma>0$, i.e.
$$
\forall\,t>0,\text{ }(X_{tu},u\geq0)\=(t^{\gamma}X_u,u\geq0).
$$
Let $q:\R\to\R$ be a non-decreasing $\mathcal{C}^1$-function such that: 
\begin{enumerate}
\item[i)]the function $x\longmapsto xq'(x)$ is non-decreasing.  
\item[ii)] For every  $\beta,t>0$ and $x\in\R$:
$$
\E\left[\exp\left(\beta\sup\limits_{0\leq u\leq1}q(t^{\gamma}X_u)\right)\right]<\infty\text{ et }
\E\left[\exp\left(\beta\inf\limits_{0\leq u\leq1}q(t^{\gamma}X_u)\right)\right]>0.
$$
\item[iii)] For every $t,u>0$, there exists $\alpha=\alpha(t,u)>1$ such that
$$
\E\left[|X_u|^{\alpha}(q')^{\alpha}(t^{\gamma}X_u)\right]<\infty.
$$
\end{enumerate}
Then, after the change of variable $s=tu$, we deduce from Theorem \ref{theo:SCMvolatilePcocNorme} that
$$
\left(N_t:=\frac{\exp\left(\dfrac{1}{t}\displaystyle\int_0^tq(X_s)ds\right)}{\E\left[\exp\left(\dfrac{1}{t}\displaystyle\int_0^tq(X_s)ds\right)\right]},t\geq0\right)\text{ is a peacock.}
$$
\end{exa}

The purpose of the next example is to exihibit peacocks of type $N^2$ using processes which are not MTP$_2$. 
\begin{exa}\label{exa:SCMnotMTP2}
Let $(X_1,X_2)$ be a random vector with values in $\{1,2,3\}^2$ and which has the law  $\Pb(X_1=i,X_2=j)=P_{ij}$, where 
\[
P=\dfrac{1}{20}\left(
\begin{array}{ccc}
3&3&1\\
3&2&2\\
1&2&3
\end{array}
\right).
\] 
Observe that $P$ is a symmetric matrix which is not TP$_2$ since 
\[
\det\left(
\begin{array}{cc}
3&3\\
3&2
\end{array}
\right)<0.
\]
But, $(X_1,X_2)$ satisfy a SCM type condition (weaker than SCM). Precisely, for every componentwise non-decreasing function $\phi:\{1,2,3\}^2\to\R$ and every non-decreasing and strictly positive functions $f_1,f_2:\{1,2,3\}\to\R_+$, the maps $K_1,\,K_2:\{1,2,3\}\to\R$ defined by
\[
K_1:i\in\{1,2,3\}\longmapsto\dfrac{\E[\phi(X_1,X_2)f_1(X_1)f_2(X_2)]|X_1=i]}{\E[f_1(X_1)f_2(X_2)|X_1=i]}
\]
and
\[
K_2:i\in\{1,2,3\}\longmapsto\dfrac{\E[\phi(X_1,X_2)f_1(X_1)f_2(X_2)]|X_2=i]}{\E[f_1(X_1)f_2(X_2)|X_2=i]} 
\]
are non-decreasing.\\
Since $P$ is symmetric, it suffices to prove that $K_1$ is  non-decreasing. Observe that, for every $i=1,2,3$,
\[
K_1(i)=\dfrac{\sum\limits_{k=1}^3\phi(i,k)f_2(k)P_{ik}}{\sum\limits_{k=1}^3f_2(k)P_{ik}}.
\] 
Moreover, to show that $K_1$ is non-decreasing, we may restrict ourselves to the functions $\phi(i,\cdot)=1_{\llbracket a,+\infty\llbracket}$ $(a=2,3)$. Precisely, it is sufficient to see that $P$ satisfies 
\begin{equation}\label{eq:exaSCMnotMTP2a1}
\dfrac{1}{P_{11}} \sum\limits_{k=2}^3f_2(k)P_{1k}\leq \dfrac{1}{P_{21}}\sum\limits_{k=2}^3f_2(k)P_{2k}\leq \dfrac{1}{P_{31}}\sum\limits_{k=2}^3f_2(k)P_{3k}
\end{equation}
and  
\begin{equation}\label{eq:exaSCMnotMTP2a2}
\dfrac{P_{13}}{\sum\limits_{k=1}^2f_2(k)P_{1k}}\leq\dfrac{P_{23}}{\sum\limits_{k=1}^2f_2(k)P_{2k}}\leq\dfrac{P_{33}}{\sum\limits_{k=1}^2f_2(k)P_{3k}}. 
\end{equation}
To obtain (\ref{eq:exaSCMnotMTP2a1}) and (\ref{eq:exaSCMnotMTP2a2}), one may remark that, for every $i\in\{1,2,3\}$ and $a\in\{2,3\}$,
\[
\dfrac{\sum\limits_{k=a}^3f_2(k)P_{ik}}{\sum\limits_{k=1}^3f_2(k)P_{ik}}=\dfrac{1}{1+\dfrac{\sum\limits_{k=1}^{a-1}f_2(k)P_{ik}}{\sum\limits_{k=a}^{3}f_2(k)P_{ik}}}.
\]
Now, since 
\begin{equation}\label{eq:exaSCMnotMTP2b}
\dfrac{1}{3}[3f_2(2)+f_2(3)]\leq\dfrac{2}{3}[f_2(2)+f_2(3)]\leq 2f_2(2)+3f_2(3)
\end{equation}
and
\[
\dfrac{1}{3f_2(1)+3f_2(2)}\leq\dfrac{2}{3f_2(1)+2f_2(2)}\leq\dfrac{3}{f_2(1)+2f_2(2)},
\]
we deduce that $K_1$ is non-decreasing. Note that
the first inequality in (\ref{eq:exaSCMnotMTP2b}) holds since $f_2$ is non-decreasing.\\
Similarly, one may prove that if $(Y_1,Y_2)$ is a random vector taking values in $\{1,2,3\}^2$ and whose law is given by $\Pb(Y_1=i,Y_2=j)=P^{\ast}_{ij}$, with
\[
P^{\ast}=\frac{1}{20}\left(
\begin{array}{ccc}
3&2&1\\
2&2&3\\
1&3&3
\end{array}
\right)
\] 
then, for every componentwise non-decreasing function $\phi:\{1,2,3\}^2\to\R$ and for every non-increasing and strictly positive functions $g_1,g_2:\{1,2,3\}\to\R$,
\[
K^{\ast}_1:i\longmapsto\dfrac{\E[\phi(Y_1,Y_2)g_1(Y_1)g_2(Y_2)|Y_1=i]}{\E[g_1(Y_1)g_2(Y_2)|Y_1=i]}
\]
and
\[
K^{\ast}_2:i\longmapsto\dfrac{\E[\phi(Y_1,Y_2)g_1(Y_1)g_2(Y_2)|Y_2=i]}{\E[g_1(Y_1)g_2(Y_2)|Y_2=i]}
\]
are non-decreasing.
\begin{corol}
If $q_1,q_2:\R_+\to\R$ are two non-decreasing $\mathcal{C}^1$ class functions such that $x\longmapsto xq^{\prime}_1(x)$ and $x\longmapsto xq^{\prime}_2(x)$ are also non-decreasing, then 
\[
\left(N_t=\dfrac{\exp(q_1(tX_1)+q_2(tX_2))}{\E[\exp(q_1(tX_1)+q_2(tX_2))]},t\geq0\right) 
\]
and
\[
\left(N^{\ast}_t=\dfrac{\exp(-q_1(tY_1)-q_2(tY_2))}{\E[\exp(-q_1(tY_1)-q_2(tY_2))]},t\geq0\right)
\]
are peacocks.
\end{corol} 
\end{exa}
\noindent
\text{}\\
\textbf{Acknowledgements: }We are grateful to Professors Francis Hirsch, Bernard Roynette and Marc Yor for fruitfull suggestions and comments. We also thank the referee for his comments and technical remarks.

\end{document}